\documentclass[10pt]{article}

\usepackage{amsthm}
\usepackage{amsmath}
\usepackage{amsfonts}
\usepackage{amssymb}
\usepackage{mathrsfs}
\usepackage{empheq}
\usepackage{stmaryrd}
\usepackage{dsfont}
\usepackage{enumerate}
\usepackage{cancel}
\usepackage{comment}
\usepackage{pgfplots}
\usepackage{soul}
\usepackage{multirow}
\usetikzlibrary{patterns} 
\newcommand{\hide}[1]{}
\usepackage[applemac]{inputenc}
\usepackage{hyperref}
\usepackage{xcolor}

\definecolor{mygreen}{rgb}{0, 0.5, 0}

\newcommand{\1}{\mathds{1}}

\newcommand{\N}{\mathbb{N}}

\newcommand{\R}{\mathbb{R}}
\newcommand{\C}{\mathbb{C}}

\newcommand{\Dr}{\mathscr{D}}

\newcommand{\Or}{\mathscr{Or}}

\newcommand{\vphi}{\varphi}
\newcommand{\eps}{\varepsilon}
\newcommand{\dsp}{\displaystyle}
\newcommand{\ovl}{\overline}

\newcommand{\vlim}{\lim\limits}

\newcommand{\vmax}{\max\limits}

\newcommand{\vsup}{\sup\limits}

\newcommand{\vint}{\int\limits}

\newcommand{\inj}{\hookrightarrow}
\newcommand{\tends}{\longrightarrow}
\newcommand{\weak}{\rightharpoonup}
\newcommand{\wt}{\widetilde}

\newcommand{\loc}{\mathrm{loc}}

\renewcommand{\b}{\mathrm{b}}
\newcommand{\co}{\mathrm{c}}
\renewcommand{\d}{\mathrm{d}}

\newcommand{\vi}{\mathrm{i}}
\newcommand{\e}{\mathrm{e}}
\newcommand{\w}{{\textsl w}}
\renewcommand{\le}{\leqslant}
\renewcommand{\ge}{\geqslant}
\renewcommand{\Re}{\mathrm{Re}}
\renewcommand{\Im}{\mathrm{Im}}
\newcommand{\bs}{\boldsymbol}
\newcommand{\p}{\prime}

\newcommand{\eqdef}{\stackrel{\mathrm{def}}{=}}
\DeclareMathOperator{\supp}{supp}

\DeclareMathOperator{\Arg}{Arg}

\numberwithin{equation}{section}

\newtheorem{thm}{Theorem}[section]
\newtheorem{prop}[thm]{Proposition}

\newtheorem{lem}[thm]{Lemma}

\theoremstyle{definition}
\newtheorem{rmk}[thm]{Remark}
\newtheorem{defi}[thm]{Definition}

\newtheorem{assum}[thm]{Assumption}

\newenvironment{proof*}{\noindent{\bf Proof.}}{\qed}
\newenvironment{vproof}[1]{\noindent{\bf Proof #1}}{\qed}

\title{\huge Damped nonlinear Ginzburg--Landau equation with saturation. Part I. Existence of solutions on general domains}
\author{\sc Pascal B\'{e}gout and Jes\'us Ildefonso D{\'{\i}}az}
\date{}

\begin{document}

\maketitle

\begin{center}
\begin{tabular}{ll}
\hspace*{-.28cm}$^*$ Toulouse School of Economics	&	\hspace*{-.25cm}$^\dagger$ Instituto de Matem\'atica Interdisciplinar	\\
Université Toulouse Capitole 						&	Universidad Complutense de Madrid								\\
Institut de Mathématiques de Toulouse 				&	Plaza de las Ciencias, 3										\\
1, Esplanade de l'Université 						&	28040 Madrid, SPAIN										\\
31080 Toulouse Cedex 6, FRANCE					& 															\\
{\footnotesize E-mail\:: \href{mailto:Pascal.Begout@math.cnrs.fr}{\texttt{Pascal.Begout@math.cnrs.fr}}}
&
{\footnotesize E-mail\:: \href{mailto:jidiaz@ucm.es}{\texttt{jidiaz@ucm.es}}}
\end{tabular}
\end{center}

\begin{abstract}
We study the complex Ginzburg--Landau equation posed on possibly unbounded domains, including some singular and saturated nonlinear damping terms. This model interpolates between the nonlinear Schr\"odinger equation and dissipative parabolic dynamics through a complex time-derivative prefactor, capturing the interplay between dispersion and dissipation. Under suitable structural conditions on the complex coefficients, we establish the existence and uniqueness of global solutions. The analysis relies on the delicate proofs that the maximal monotone operator theory can be adapted to this framework, even for unbounded domains.
\end{abstract}

{\let\thefootnote\relax\footnotetext{Pascal Bégout acknowledges funding from ANR under grant ANR-17-EUR-0010 (Investissements d'Avenir program)}}
{\let\thefootnote\relax\footnotetext{The research of J.\,I.\:D\'{\i}az was partially supported by the project PID-2020-112517GBI00 of the AEI and MCIU/AEI/10.13039/-501100011033/FEDER, EU}}
{\let\thefootnote\relax\footnotetext{$^*$\href{https://orcid.org/0000-0002-9172-3057}{https://orcid.org/0000-0002-9172-3057}}}
{\let\thefootnote\relax\footnotetext{$^\dagger$\href{https://orcid.org/0000-0003-1730-9509}{https://orcid.org/0000-0003-1730-9509}}}
{\let\thefootnote\relax}
{\let\thefootnote\relax}

\tableofcontents

\baselineskip .6cm

\section{Introduction}
\label{introduction}

The complex Ginzburg--Landau equation constitutes one of the most fundamental models in the theory of nonlinear dissipative systems. Originally introduced in the context of superconductivity, it describes the evolution of a complex order parameter whose modulus represents the density of superconducting electron pairs, while its phase encodes macroscopic quantum coherence. Beyond this original setting, Ginzburg--Landau type equations arise naturally in nonlinear optics, laser dynamics, superfluidity, Bose--Einstein condensation, chemical reactions, and pattern formation in systems far from thermodynamic equilibrium (see, for instance, the pioneering work of Ginzburg and Landau~\cite{zbMATH07887240} in 1950 in superconductivity, the monographs by Kuramoto~\cite{MR762432}, Temam~\cite{MR1441312} and Levy~\cite{zbMATH01456176} and some survey papers as, e.g., Battogtokh and Mikhailov~\cite{MR1369492}, Aranson and Kramer~\cite{MR1895097}).

\bigskip
\noindent
From a physical point of view, the relevance of the complex Ginzburg--Landau equation lies in its ability to capture the coexistence and competition between diffusion, dispersion, nonlinear amplification, saturation and dissipation. This mixture of mechanisms makes it a universal amplitude equation governing the modulation of instabilities near criticality and the emergence of coherent structures such as vortices, defects and spatio-temporal patterns.

\bigskip
\noindent
From a mathematical perspective, the equation provides a paradigmatic example of an infinite-dimensional dissipative dynamical system. As emphasized in the monograph of Temam~\cite{MR1441312} (see also Ginibre and Velo~\cite{MR1406282,MR1463822}), it plays a central role in the development of the theory of global attractors, asymptotic compactness and long-time dynamics for nonlinear evolution equations.

\bigskip
\noindent
In several recent works (see, e.g., \cite{MR4725781}), the strong stabilization of a damped nonlinear Schr\"{o}dinger equation with saturation effects was established on unbounded domains. That analysis demonstrates that suitably chosen nonlinear damping mechanisms can overcome dispersive
effects even in the absence of compactness properties typically available in bounded domains. Such results are particularly relevant for physical systems
modeled in open space, where boundary confinement cannot be assumed. The main goal of this paper is to extend the general approach taken in the
theory presented in \cite{MR4725781} in order to extend previous results in the literature on complex Ginzburg--Landau equation in which the
saturation term is understood as an absorption term (see, e.g., Antontsev, Dias and Figueira~\cite{MR3208711} and~\cite{zbMATH05502727,MR2194979,MR2268809,MR2365580,MR4098331}).

\bigskip
\noindent
The damped nonlinear Schr\"{o}dinger equation may be viewed as a limiting or simplified model within the broader Ginzburg--Landau framework. Introducing a complex coefficient in front of the time derivative allows one to interpolate continuously between purely dispersive Schr\"{o}dinger dynamics and purely dissipative parabolic dynamics. This observation motivates the extension of the stabilization theory developed in \cite{MR4725781} to the complex Ginzburg--Landau equation posed on general domains $\Omega\subseteq\R^N$ (possible unbounded), with boundary $\partial\Omega,$
\begin{empheq}[left=\empheqlbrace]{align}
\label{gl}
\e^{-\vi\theta}\frac{\partial u}{\partial t}-\Delta u+a|u|^{-(1-m)}u+b|u|^{p-1}u+\gamma u=f,	&	\text{ in } (0,\infty)\times\Omega,				\\
\label{glb}
u_{|\partial\Omega}=0,													&	\text{ on } (0,\infty)\times\partial\Omega,	\dfrac{}{}	\\
\label{u0}
u(0)= u_0,																&	\text{ in } \Omega,
\end{empheq}
where $\theta\in\left(-\frac\pi2,\frac\pi2\right),$ $0\le m\le1$ and $a,b,\gamma\in\C.$ Here we write, for generality $p\in(1,\infty)$ but the physically more often case considered in the literature corresponds to $p=3.$

\bigskip
\noindent
For $\theta\in\left[-\frac\pi2,\frac\pi2\right]$ and $m\ge0,$ we introduce the following set of complex numbers:
\begin{gather}
\label{Cm}
C_\theta(m)=\Big\{z\in\C; \; \Re(z\e^{\vi\theta})>0 \text{ and } 2\sqrt m\,\Re(z\e^{\vi\theta})\ge|1-m|\,|\Im(z\e^{\vi\theta})|\Big\}.
\end{gather}
In the particular cases in which $m\in\{0,1\},$ the set $C_\theta(m)$ becomes,
\begin{align}
\label{C0}
& C_\theta(0)=\Big\{z\in\C; \; \Re(z\e^{\vi\theta})>0 \text{ and } \Im(z\e^{\vi\theta})=0\Big\},	\\
\label{C1}
& C_\theta(1)=\Big\{z\in\C; \; \Re(z\e^{\vi\theta})>0\Big\},
\end{align}
and actually,
\begin{gather}
\label{C0'}
C_\theta(0)=\Big\{z\in\C; \; \exists\mu>0 \text{ such that } z=\mu\,\e^{-\vi\theta}\Big\}.
\end{gather}

\noindent
We note that if $\theta=\dfrac\pi2,$ $0\le m\le1,$ $a\in C_\theta(m),$ $b=0,$ $\gamma=-V(x)\in L^1_\loc(\Omega;\R)$ and $f\in L^1_\loc\big([0,\infty);L^2(\Omega)\big),$ then equation~\eqref{gl} becomes
\begin{gather}
\label{nls}
\vi\frac{\partial u}{\partial t}+\Delta u+V(x)u-a|u|^{-(1-m)}u=-f.
\end{gather}
It follows that the nonlinear Schr\"{o}dinger equation~\eqref{nls} is a limit case of the Ginzburg--Landau equation~\eqref{gl}, in terms of $\theta.$ But the Ginzburg--Landau equation~\eqref{gl} may also be considered as an intermediate equation between the nonlinear Schr\"{o}dinger equation and the nonlinear heat equation
\begin{gather*}
\frac{\partial u}{\partial t}-\Delta u+a|u|^{-(1-m)}u=f,
\end{gather*}
by taking $\theta=0,$ $a\in\R$ and $b=\gamma=0$ in \eqref{gl}. In this last case, $a\in C_0(m)$ only means that $a$ is a positive real number.

\bigskip
\noindent
We point out that the analysis of the complex Ginzburg--Landau equation on unbounded domains presents several intertwined difficulties. The complex prefactor $e^{-\vi\theta}$ in front of the time derivative breaks both the purely Hamiltonian structure of Schr\"{o}dinger-type equations and the gradient-flow structure of parabolic equations.

\bigskip
\noindent
The strategy of the proof relies on the use of suitable energy methods, sharpening the ones presented in the monograph \cite{MR2002i35001}. Those methods capture the effective dissipation induced by the nonlinear terms, combined with refined energy estimates adapted to unbounded domains. Singular nonlinearities with $0\le m<1$ require weak formulations and truncation arguments to control the dynamics near vanishing amplitudes. In a different paper, \cite{Part2}, we will prove the strong stabilization of solutions in a finite time under suitable conditions. 

\bigskip
\noindent
The nonlinear terms appearing in the complex Ginzburg--Landau equation introduce amplitude-dependent dissipation that becomes particularly effective in regimes where linear mechanisms fail. The singular term $|u|^{-(1-m)}u$ acts as a strong damping mechanism near low-amplitude states, suppressing residual oscillations and preventing the persistence of small-amplitude coherent structures. From the physical point of view, this term can be interpreted as a saturation or threshold effect that inhibits the survival of weak excitations.

\bigskip
\noindent
The complex prefactor $e^{-\vi\theta}$ plays a fundamental role in shaping the dynamics. For $\theta\neq0$, the system no longer conserves energy in the Hamiltonian sense, and the interaction between dispersive and dissipative components leads to a gradual relaxation toward equilibrium. This behavior
is characteristic of systems far from equilibrium, where dissipation and dispersion coexist and compete.

\bigskip
\noindent
In this paper, existence and uniqueness of the solutions are obtained under the assumption that $a\in C_\theta(m),$ while for the equation~\eqref{nls}, they are proved in the series of papers \cite{MR4098330,MR4053613,MR4340780,MR4503241,MR4725781} under the assumption that $-a\in C(m),$ where
\begin{gather*}
C(m)=\Big\{z\in\C; \; \Im(z)>0 \text{ and } 2\sqrt m\Im(z)\ge(1-m)|\Re(z)|\Big\}.
\end{gather*}
We note that assuming $a\in C_\frac\pi2(m)$ is equivalent to assuming $-a\in C(m).$

\bigskip
\noindent
The organization of this paper is the following. Section~\ref{exiuni} presents the statements of the main results concerning the existence and uniqueness of the solutions. Before to entering into the detailed proofs, we collect in Section~\ref{notfunana} a set of notations and some basic results of Functional Analysis which will be used in this paper. Although several other alternatives are possible, in our approach to the proofs of the existence of solutions we will use the abstract theory of maximal monotone operators, according the well-known theory mainly developed by Ha\"{\i}m Brezis and improving the results presented in Okazawa and Yokota~\cite{MR1886827}. Section ~\ref{maxmon} contains the detailed definition of the operators we will consider proving their monotonicity and the range properties which show that they are maximal monotone operators. Finally, in Section~\ref{proofexi}, we present the proofs of the results concerning the existence and uniqueness of the solutions.

\section{Existence and uniqueness of the solutions}
\label{exiuni}

Our main assumptions concerning the existence and uniqueness of the solutions are the following:
\begin{assum}
\label{ass}
We assume the following.
\begin{gather}
\label{O}
\Omega \text{ is any nonempty open subset of } \R^N,	\\
\label{t}
-\dfrac\pi2<\theta<\dfrac\pi2,						\\
\label{m}
m\in[0,1] \; \text{ and } \; p\in(1,\infty),					\\
\label{a}
a\in C_\theta(m) \; \text{ and } \; b\in C_\theta(p)\cup\{0\},	\\
\label{g}
\gamma\in\C \text{ with } \Re(\gamma\e^{\vi\theta})\ge0,
\end{gather}
where $C_\theta$ is defined in~\eqref{Cm}.
\end{assum}

\begin{defi}
\label{defsatsec}
Let $\Or\subseteq\R^N$ be an open subset and let $u\in L^1_\loc(\Or).$ A function $U$ is said to be a \textit{saturated section} associated to $u$ if $U\in L^\infty(\Or),$ $\|U\|_{L^\infty(\Or)}\le1$ and $U=\dfrac{u}{|u|},$ almost everywhere where $u\neq0.$
\end{defi}

\noindent
Now, let us define the notion of solution.

\begin{defi}
\label{defsol}
Assume~\eqref{O} and \eqref{m}. Let $\theta\in[0,2\pi],$ $a,b,\gamma\in\C,$ $f\in L^1_\loc\big([0,\infty);L^2(\Omega)\big)$ and $u_0\in L^2(\Omega).$ Let us consider the following assertions.
\begin{enumerate}
\item
\label{defsol1}
Let $X_{m,p}=L^{m+1}(\Omega)\cap L^{p+1}(\Omega).$ For $q=m$ or $q=p,$ we have that
\begin{gather}
\label{defsol11}
u\in L^{q+1}_\loc\big([0,\infty);H^1_0(\Omega)\cap X_{m,p}\big)\cap W^{1,\frac{q+1}q}_\loc\big([0,\infty);H^\star+X_{m,p}^\star\big)
\end{gather}
\item
\label{defsol2}
For almost every $t>0,$ $\Delta u(t)\in H^\star.$
\item
\label{defsol3}
\begin{enumerate}
\item
\label{defsol3a}
If $m>0$ then $u$ satisfies~\eqref{gl} in $\Dr^\p\big((0,\infty)\times\Omega\big).$
\item
\label{defsol3b}
If $m=0$ then there exists a saturated section $U$ associated to $u$ such that the pair $(u,U)$ satisfies
\begin{gather}
\label{gl0}
\e^{-\vi\theta}\frac{\partial u}{\partial t}-\Delta u+a\,U+b|u|^{p-1}u+\gamma u=f, \text{ in } \Dr^\p\big((0,\infty)\times\Omega\big).
\end{gather}
\end{enumerate}
\item
\label{defsol4}
We have that $u(0)=u_0,$ in $L^2(\Omega).$
\end{enumerate}
We shall say that $u$ is a \textit{strong solution} if $u$ is an $H^2$-solution or an $H^1_0$-solution. We shall say that $u$ is an $H^2$-\textit{solution} (respectively, an $H^1_0$-\textit{solution}) \textit{to} \eqref{gl}--\eqref{u0}, if $u$ satisfies the Assertions~\ref{defsol1}--\ref{defsol4} with $H=L^2(\Omega)$ \big(respectively, with $H=H^1_0(\Omega)\big).$
\\
We shall say that $u$ is an $L^2$-\textit{solution} or a \textit{weak solution to} \eqref{gl}--\eqref{u0} if there exists a pair,
\begin{gather}
\label{fn}
(u_n,f_n)_{n\in\N}\subset C\big([0,\infty);L^2(\Omega)\big)\times L^1_\loc\big([0,\infty);L^2(\Omega)\big),
\end{gather}
such that for any $n\in\N,$ $u_n$ is an $H^2$-solution to \eqref{gl}--\eqref{u0} where the right hand side of \eqref{gl} is $f_n,$ and if
\begin{gather}
\label{cv}
u_n\xrightarrow[n\to\infty]{C([0,T];L^2(\Omega))}u \; \text{ and } \; f_n\xrightarrow[n\to\infty]{L^1((0,T);L^2(\Omega))}f,
\end{gather}
for any $T>0.$ Sometimes, we shall write $(u,f),$ $(u,U)$ or $(u,U,f)$ to designate a solution with the obvious meanings.
\end{defi}

\begin{rmk}
\label{rmkdefsol}
Below are some comments about Definition~\ref{defsol}.
\begin{enumerate}
\item
\label{rmkdefsol1}
If $a\in C_\theta(0)$ then by \eqref{C0'} there exists $\mu>0$ such that $a=\mu\e^{-\vi\theta}.$ It follows that \eqref{gl0} may be rewritten as,
\begin{gather}
\label{gl0'}
\e^{-\vi\theta}\frac{\partial u}{\partial t}-\Delta u+\mu\,\e^{-\vi\theta}U+b|u|^{p-1}u+\gamma u=f, \text{ in } \Dr^\p\big((0,\infty)\times\Omega\big).
\end{gather}
\item
\label{rmkdefsol2}
The assumption~\eqref{defsol11} is made to have that any solution belong to $C\big([0,\infty);L^2(\Omega)\big).$ Indeed, we have by \cite[Lemma~A.4]{MR4053613} and \cite[Theorem~5.3]{MR4725781} that for any $q\in[0,\infty],$
\begin{gather*}
L^{q+1}_\loc\big([0,\infty);H^1_0(\Omega)\cap X_{m,p}\big)
\cap W^{1,\frac{q+1}q}_\loc\big([0,\infty);H^{-1}(\Omega)+X_{m,p}^\star\big)\inj C\big([0,\infty);L^2(\Omega)\big),
\end{gather*}
and then any strong solution belongs to $C\big([0,\infty);L^2(\Omega)\big).$ As a consequence, \eqref{fn} and \ref{defsol4} of Definition~\ref{defsol} make sense and any weak solution also belongs to $C\big([0,\infty);L^2(\Omega)\big).$
\item
\label{rmkdefsol3}
Assume $m=0.$ Let $(u,U)$ be a weak solution. It is clear from \eqref{gl0} that if $u$ is unique in $\Dr^\p\big((0,\infty)\times\Omega\big)$ then so is $U.$
\item
\label{rmkdefsol4}
We easily deduce from Definition~\ref{defsatsec} and \eqref{cv} that if $u$ is a weak solution then each term of equation \eqref{gl} belongs to $\Dr^\p\big((0,\infty)\times\Omega\big),$ except possibly $b|u|^{p-1}u.$ It follows that the existence of a strong solution implies the existence of a weak solution when $b=0.$ Otherwise, we have to show that $b|u|^{p-1}u\in\Dr^\p\big((0,\infty)\times\Omega\big).$
\end{enumerate}
\medskip
\end{rmk}

\noindent
The goal of the next result is to clarify in which way the weak solutions satisfy the equation~\eqref{gl} in the unsaturated case and without the presence of the superlinear term $(m\neq0$ and $b=0).$ This permits to give a result without assumption on the parameters $\theta,$ $a$ and $\gamma.$ Note that if $\theta=\pm\frac\pi2$ then equation \eqref{gl} becomes the nonlinear Schr\"{o}dinger equation and we recover the result by \cite{MR4503241} (Proposition~2.5). Nevertheless, in the presence of this superlinearity and in the unsaturated case $(b\neq0$ and $m\neq0),$ we may give a sense to the equation~\eqref{gl} better than in $\Dr^\p\big((0,\infty)\times\Omega\big)$ with the additional assumptions \eqref{t}, \eqref{a} and \eqref{g}. A weaker statement holds in the saturated case. See Theorems~\ref{thmweaksol} and \ref{thmweaksols} below.

\begin{prop}[\textbf{The unsaturated case}]
\label{propsolL2}
Let $\Omega\subseteq\R^N$ be an open subset, $\theta\in\left[-\frac\pi2,\frac\pi2\right],$ $0<m\le1,$ $a,\gamma\in\C,$ $b=0$ and $f\in L^1_\loc\big([0,\infty);L^2(\Omega)\big).$ Let $u$ be a weak solution to \eqref{gl}. Let $(u_n,f_n)_{n\in\N}$ be any sequence of $H^2$-strong solutions to~\eqref{gl}--\eqref{glb} satisfying~\eqref{cv}. Then,
\begin{gather}
\label{propsolL21}
u\in W^{1,1}_\loc\big([0,\infty);H^{-2}(\Omega)+L^\frac2m(\Omega)\big),
\end{gather}
and $u$ solves~\eqref{gl} in $L^1_\loc\big([0,\infty);H^{-2}(\Omega)+L^\frac2m(\Omega)\big)$ and so in $\Dr^\p\big((0,\infty)\times\Omega\big).$ In addition,
\begin{gather}
\label{propsolL22}
u_n\xrightarrow[n\to\infty]{W^{1,1}((0,T);H^{-2}(\Omega)+L^\frac2m(\Omega))}u.
\end{gather}
for any $T>0.$
\end{prop}

\begin{prop}[\textbf{Uniqueness and continuous dependance}]
\label{propdep}
Let Assumption~$\ref{ass}$ be fulfilled, let $f,\wt f\in L^1_\loc\big([0,\infty);L^2(\Omega)\big)$ and $X=H^1_0(\Omega)\cap L^{m+1}(\Omega)\cap L^{p+1}(\Omega).$ Finally, let $1\le q\le\infty$ and let
\begin{gather}
\label{propdep1}
u,\wt u\in L^q_\loc\big([0,\infty);X\big)\cap W^{1,q^\p}_\loc\big([0,\infty);X^\star\big)\inj C\big([0,\infty);L^2(\Omega)\big),
\end{gather}
be solutions in $\Dr^\p\big((0,\infty)\times\Omega\big)$ to,
\begin{gather*}
\e^{-\vi\theta}u_t-\Delta u+a|u|^{-(1-m)}u+b|u|^{p-1}u+\gamma u=f,						\\
\e^{-\vi\theta}\wt {u_t}-\Delta\wt u+a|\wt u|^{-(1-m)}\wt u+b|\wt u|^{p-1}\wt u+\gamma\wt u=\wt f ,
\end{gather*}
respectively. Then,
\begin{gather}
\label{propdep2}
\|u(t)-\wt u(t)\|_{L^2(\Omega)}\le\|u(s)-\wt u(s)\|_{L^2(\Omega)}+\vint_s^t\|f(\sigma)-\wt f(\sigma)\|_{L^2(\Omega)}\d\sigma,
\end{gather}
for any $t\ge s\ge0.$ Finally,~\eqref{propdep2} also holds true for the weak solutions.
\end{prop}

\begin{rmk}
\label{rmkpropdep}
If $m=0$ then it is understood in Proposition~\ref{propdep} that $|u|^{-(1-m)}u=U,$ where $U$ is a saturated section associated to $u.$
\end{rmk}

\begin{thm}[\textbf{Existence and uniqueness of $\bs{L^2}$-solutions}]
\label{thmweak}
Let Assumption~$\ref{ass}$ be fulfilled and let $f\in L^1_\loc\big([0,\infty);L^2(\Omega)\big).$ For any $u_0\in L^2(\Omega),$ there exists a unique weak solution to \eqref{gl}--\eqref{u0}.
\end{thm}

\begin{thm}[\textbf{Properties of the $\bs{L^2}$-solutions}]
\label{thmweaksol}
Let Assumption~$\ref{ass}$ be fulfilled, let $u_0\in L^2(\Omega),$ $f\in L^2_\loc\big([0,\infty);L^2(\Omega)\big)$ and $u$ be the unique weak solution to~\eqref{gl}--\eqref{u0} given by Theorem~$\ref{thmweak}.$ Then,
\begin{gather}
\label{H10}
u\in L^2_\loc\big([0,\infty);H^1_0(\Omega)\big),	\\
\label{Lm}
u\in L^{m+1}_\loc\big([0,\infty);L^{m+1}(\Omega)\big)\cap L^{p+1}_\loc\big([0,\infty);L^{p+1}(\Omega)\big),	
\end{gather}
and for any $t\ge s\ge0,$
\begin{gather}
\label{L2+}
\begin{split}
&\dfrac12\|u(t)\|_{L^2(\Omega)}^2+\cos\theta\vint_s^t\|\nabla u(\sigma)\|_{L^2(\Omega)}^2\d\sigma
+\Re(a\e^{\vi\theta})\vint_s^t\|u(\sigma)\|_{L^{m+1}(\Omega)}^{m+1}\d\sigma+\Re(b\e^{\vi\theta})\vint_s^t\|u(\sigma)\|_{L^{p+1}(\Omega)}^{p+1}\d\sigma \\
&+\Re(\gamma\e^{\vi\theta})\vint_s^t\|u(\sigma)\|_{L^2(\Omega)}^2\d\sigma		
\le\dfrac12\|u(s)\|_{L^2(\Omega)}^2+\Re\left(\e^{\vi\theta}\iint\limits_{s\;\Omega}^{\text{}\;\;t}f(\sigma,x)\,\ovl{u(\sigma,x)}\,\d x\,\d\sigma\right),
\end{split}
\end{gather}
Finally, setting $Y_{m,p}=H^1_0(\Omega)\cap L^{m+1}(\Omega)\cap L^{p+1}(\Omega),$ if $m>0$ and if $(u_n,f_n)_{n\in\N}$ is any sequence of $H^2$-strong solutions to~\eqref{gl}--\eqref{glb} satisfying~\eqref{cv} then,
\begin{gather}
\label{W11}
u\in W^{1,1}_\loc\big([0,\infty);Y_{m,p}^\star\big),	\\
\label{conL1w}
\frac{\partial u_n}{\partial t}\overset{L^1_\loc([0,\infty);Y_{m,p}^\star)_\w}
{\underset{n\tends\infty}{-\!\!\!-\!\!\!-\!\!\!-\!\!\!-\!\!\!-\!\!\!-\!\!\!-\!\!\!-\!\!\!-\!\!\!-\!\!\!-\!\!\!\weak}}\frac{\partial u}{\partial t},	\\
\label{senseq}
u \text{ satisfies~\eqref{gl} in } L^1_\loc\big([0,\infty);Y_{m,p}^\star\big)\inj\Dr^\p\big((0,\infty)\times\Omega\big).
\end{gather}
\end{thm}

\noindent
Due to the lack of separability of $L^\infty(\Omega),$ we encounter some difficulties about the measurability of the saturated nonlinearity $U:[0,\infty)\tends L^\infty(\Omega).$ To avoid this, we introduce a sequence $(Y_n)_{n\in\N_0}$ of $L^1$-approximating sequence of RNP-spaces: $(Y_n)_{n\in\N_0}$ is a sequence of Banach spaces such that for any $n\in\N_0,$ $\Dr(\Omega)\inj Y_n\inj Y_{n+1}\inj L^1(\Omega)$ with dense embeddings, $Y_n$ is separable and reflexive when $n\ge1,$ $Y_n\neq Y_{n+1}\neq L^1(\Omega)$ and for any $f\in Y_0,$ $\|f\|_{L^1(\Omega)}\le\|f\|_{Y_n}$ and $\vlim_{n\to\infty}\|f\|_{Y_n}=\|f\|_{L^1(\Omega)}.$ This sequence exists with the help of \cite[Lemma~5.8]{MR4725781}. For more details, see \cite[Section~5]{MR4725781}.

\begin{thm}[\textbf{The saturated case}]
\label{thmweaksols}
Let Assumption~$\ref{ass}$ be fulfilled with $m=0,$ let $u_0\in L^2(\Omega),$ let $f\in L^2_\loc\big([0,\infty);L^2(\Omega)\big)$ and let $(u,U)$ be the unique weak solution to~\eqref{gl}--\eqref{u0} given by Theorem~$\ref{thmweak}.$ Let $(Y_\ell)_{\ell\in\N_0}$ be any $L^1$-approximating sequence of RNP-spaces. Finally, let $(u_n,U_n,f_n)_{n\in\N}$ be any sequence of $H^2$-strong solutions to~\eqref{gl}--\eqref{glb} satisfying~\eqref{cv}. Then for any $\ell\in\N_0,$
\begin{gather}
\label{W11l}
u\in W^{1,1}_\loc\big([0,\infty);Z_{\ell,p}^\star\big),	\\
\label{conL1lw}
\frac{\partial u_n}{\partial t}\overset{L^1((0,T);Z_{\ell,p}^\star)_\w}
{\underset{n\tends\infty}{-\!\!\!-\!\!\!-\!\!\!-\!\!\!-\!\!\!-\!\!\!-\!\!\!-\!\!\!-\!\!\!-\!\!\!-\!\!\!-\!\!\!\weak}}\frac{\partial u}{\partial t},	\\
\label{senseqs}
(u,U) \text{ satisfies~\eqref{gl} in } L^1_\loc\big([0,\infty);Z_{\ell,p}^\star\big)\inj\Dr^\p\big((0,\infty)\times\Omega\big),
\end{gather}
for any $T>0,$ where $Z_{\ell,p}=H^1_0(\Omega)\cap Y_\ell\cap L^{p+1}(\Omega).$
In addition,
\begin{gather}
\label{thmweaksols1}
U_n\underset{n\to\infty}{\overset{L^\infty((0,T)\times\Omega)_{\w\star}}{-\!\!\!-\!\!\!-\!\!\!-\!\!\!-\!\!\!-\!\!\!-\!\!\!-\!\!\!-\!\!\!-\!\!\!-\!\!\!\weak}}U,
\end{gather}
for any $T>0.$ Finally there exists $N_0\subset(0,\infty)$ with $|N_0|=0$ such that for any $t\in(0,\infty)\cap N_0^\co,$
\begin{gather}
\label{thmweaksols2}
u^\p(t)\in H^{-1}(\Omega)+L^\infty(\Omega)+L^\frac{p+1}p(\Omega).
\end{gather}
In particular, equation~\eqref{gl} makes sense in $H^{-1}(\Omega)+L^\infty(\Omega)+L^\frac{p+1}p(\Omega),$ for almost every $t>0.$
\end{thm}

\begin{rmk}
\label{rmkpropsolL2}
Below are some comments about Theorem~\ref{thmweaksols}.
\begin{enumerate}
\item
\label{rmkpropsolL21}
If $|\Omega|<\infty$ then the spaces $Z_{\ell,p}^\star$ may be replaced with $H^{-1}(\Omega)+L^\frac{p+1}p(\Omega).$ See the end of the proof of Theorem~\ref{thmweaksols} for the details.
\item
\label{rmkpropsolL22}
Whether or not $u^\p:(0,\infty)\tends H^{-1}(\Omega)+L^\infty(\Omega)+L^\frac{p+1}p(\Omega)$ is measurable is an open question.
\item
\label{rmkpropsolL23}
It follows from the properties of $(Y_\ell)_{\ell\in\N_0}$ that for any $\ell\in\N_0,$
\begin{gather*}
\|U\|_{L^\infty((0,\infty);Y_\ell^\star)}\le\|U\|_{L^\infty((0,\infty)\times\Omega)}\le1,
\end{gather*}
\end{enumerate}
\end{rmk}

\begin{thm}[\textbf{Additional regularity of the weak solutions}]
\label{thmstrongaH1}
Let Assumption~$\ref{ass}$ be fulfilled, let $f\in L^2_\loc\big([0,\infty);L^2(\Omega)\big),$ let $u_0\in H^1_0(\Omega)$ and let $u$ be the unique weak solution to~\eqref{gl}--\eqref{u0} given by Theorem~$\ref{thmweak}.$ Then we have that,
\begin{gather}
\begin{cases}
\label{thmstrongaH11}
u\in C_\w\big([0,\infty);H^1_0(\Omega)\big),							\medskip \\
\Delta u\in L^2_\loc\big([0,\infty);L^2(\Omega)\big).
\end{cases}
\end{gather}
In addition, for any $t\ge s\ge0.,$
\begin{gather}
\label{thmstrongaH12}
\|\nabla u(t)\|_{L^2(\Omega)}^2+\cos\theta\vint_s^t\|\Delta u(\sigma)\|_{L^2(\Omega)}^2\d\sigma
\le\|\nabla u(s)\|_{L^2(\Omega)}^2+\frac1{\cos\theta}\vint_s^t\|f(\sigma)\|_{L^2(\Omega)}^2\d\sigma,
\end{gather}
\textbf{Case }$\bs{m>0.}$
Then $u\in W^{1,q}_\loc\big([0,\infty);L^2(\Omega)+L^\frac{m+1}m(\Omega)+L^\frac{p+1}p(\Omega)\big)$ and $u$ satisfies \eqref{gl} in
\begin{gather*}
L^q_\loc\big([0,\infty);L^2(\Omega)+L^\frac{m+1}m(\Omega)+L^\frac{p+1}p(\Omega)\big)
\text{ where } q=
\begin{cases}
2,			&	\text{if } b=0,	\\
\frac{p+1}p,	&	\text{if } b\neq0.
\end{cases}
\end{gather*}
\textbf{Case }$\bs{m=0.}$ Assume further that $|\Omega|<\infty,$ let $q$ be as above and let $U$ be the saturated section associated to $u.$ Then
\begin{gather}
\label{thmstrongaH13}
u\in W^{1,q}_\loc\big([0,\infty);L^2(\Omega)+L^\frac{p+1}p(\Omega)\big), \quad U\in L^\infty_\loc\big([0,\infty);L^2(\Omega)\big),
\end{gather}
and $u$ satisfies \eqref{gl} in $L^q_\loc\big([0,\infty);L^2(\Omega)+L^\frac{p+1}p(\Omega)\big).$ Finally, the map $t\longmapsto\|u(t)\|_{L^2(\Omega)}^2$ belongs to $W^{1,1}_\loc\big([0,\infty);\R\big)$ and for almost every $t>0,$
\begin{gather}
\label{L2}
\begin{split}
	& \frac12\frac{\d}{\d t}\|u(t)\|_{L^2(\Omega)}^2+\cos\theta\|\nabla u(t)\|_{L^2(\Omega)}^2+\Re(a\e^{\vi\theta})\|u(t)\|_{L^{m+1}(\Omega)}^{m+1}	\\
  +\;	& \Re(b\e^{\vi\theta})\|u(t)\|_{L^{p+1}(\Omega)}^{p+1}+\Re(\gamma\e^{\vi\theta})\|u(t)\|_{L^2(\Omega)}^2=\Re\left(\e^{\vi\theta}\vint_{\Omega}f(t,x)\,\ovl{u(t,x)}\,\d x\right),
\end{split}
\end{gather}
\end{thm}

\begin{thm}[\textbf{Existence and uniqueness of $\bs{H^1_0}$-solutions}]
\label{thmstrongH1}
Let Assumption~$\ref{ass}$ be fulfilled and assume that one of the following assumptions holds.
\begin{empheq}[left=\empheqlbrace]{align}
\label{thmstrongH111}
&	b=0,			\dfrac{}{}	\\
\label{thmstrongH112}
&	p(N-2)\le N+2,	\dfrac{}{}	\\
\label{thmstrongH113}
&	|\Omega|<\infty.
\end{empheq}
If $m=0$ then assume further that $|\Omega|<\infty.$ Finally, let
\begin{gather}
\label{thmstrongH11}
f\in L^2_\loc\big([0,\infty);L^2(\Omega)\big)\cap L^\frac{m+1}m_\loc\big([0,\infty);H^{-1}(\Omega)+L^\frac{m+1}m(\Omega)+L^\frac{p+1}p(\Omega)\big).
\end{gather}
For any $u_0\in H^1_0(\Omega),$ there exists a unique $H^1_0$-solution $u,$ and $u$ satisfies \eqref{defsol11} with
\begin{gather*}
\begin{cases}
q=m,		&	\text{if } \eqref{thmstrongH111} \text{ or } \eqref{thmstrongH112} \text{ is satisfied},	\\
q=p,		&	 \text{if } \eqref{thmstrongH113} \text{ is satisfied}.
\end{cases}
\end{gather*}
Finally, $u$ is also a weak solution, the map $t\longmapsto\|u(t)\|_{L^2(\Omega)}^2$ belongs to $W^{1,1}_\loc\big([0,\infty);\R\big)$ and \eqref{L2} holds for almost every $t>0.$
\end{thm}

\begin{thm}[\textbf{Existence and uniqueness of $\bs{H^2}$-solutions}]
\label{thmstrongH2}
Let Assumption~$\ref{ass}$ be fulfilled and $f\in W^{1,1}_\loc\big([0,\infty);L^2(\Omega)\big).$ Then for any $u_0\in H^1_0(\Omega)\cap L^{2m}(\Omega)\cap L^{2p}(\Omega)$ with $\Delta u_0\in L^2(\Omega),$ there exists a unique $H^2$-solution $u$ to \eqref{gl}--\eqref{u0}. Furthermore, $u$ satisfies~\eqref{gl} in $L^\infty_\loc\big([0,\infty);L^2(\Omega)\big)$ and the following properties.
\begin{enumerate}
\item
\label{thmstrongH21}
We have that,
\begin{gather*}
u\in C\big([0,\infty);H^1_0(\Omega)\big)\cap W^{1,\infty}_\loc\big([0,\infty);L^2(\Omega)\big)\cap L^\infty_\loc\big([0,\infty);L^{2p}(\Omega)\big),
\end{gather*}
and, in addition, if $m>0$ then $u\in L^\infty_\loc\big([0,\infty);L^{2m}(\Omega)\big).$ If $m=0$ then the saturated section $U$ associated to $u$ satisfies \eqref{thmstrongaH13}.
\item
\label{thmstrongH22}
We have $\Delta u\in C_\w\big([0,\infty);L^2(\Omega)\big)$ and,
\begin{empheq}[left=\empheqlbrace]{align}
\label{strongH21}
	&	\|u(t)-u(s)\|_{L^2(\Omega)}\le\|u_t\|_{L^\infty((s,t);L^2(\Omega))}|t-s|,			\frac{}{}			\\
\label{strongH22}
	&	\|\nabla u(t)-\nabla u(s)\|_{L^2(\Omega)}\le M|t-s|^\frac12,									\\
\label{strongH23}
	&	\left\|u_t\right\|_{L^\infty((0,t);L^2(\Omega))}
		\le\|\e^{-\vi\theta}A_0^mu_0-f(0)\|_{L^2(\Omega)}+\int_0^t\|f^\p(\sigma)\|_{L^2(\Omega)}\d\sigma,
\end{empheq}
for any $t\ge s\ge0,$ where $M^2=2\|u_t\|_{L^\infty((s,t);L^2(\Omega))}\|\Delta u\|_{L^\infty((s,t);L^2(\Omega))}$ and $\e^{-\vi\theta}A_0^mu_0=-\Delta u_0+a|u_0|^{m-1}u_0+b|u_0|^{p-1}u_0+\gamma u_0$ $\big(\e^{-\vi\theta}A_0^0u_0=-\Delta u_0+a\frac{u_0}{|u_0|}\1_{\{u_0\neq0\}}+b|u_0|^{p-1}u_0+\gamma u_0,$ if $m=0).$
\item
\label{thmstrongH23}
The map $t\longmapsto\|u(t)\|_{L^2(\Omega)}^2$ belongs to $W^{1,\infty}_\loc\big([0,\infty);\R\big)$ and~\eqref{L2} holds for almost every $t>0.$
\item
\label{thmstrongH24}
The map $t\longmapsto\|\nabla u(t)\|_{L^2(\Omega)}^2$ belongs to $W^{1,\infty}_\loc\big([0,\infty);\R\big)$ and
\begin{gather*}
\frac{\d}{\d t}\|\nabla u(t)\|_{L^2(\Omega)}^2+\cos\theta\|\Delta u(t)\|_{L^2(\Omega)}^2\le\frac1{\cos\theta}\|f(t)\|_{L^2(\Omega)}^2,
\end{gather*}
for almost every $t>0.$
\item
\label{thmstrongH25}
If $f\in W^{1,1}\big((0,\infty);L^2(\Omega)\big)$ then $\Delta u\in L^\infty\big((0,\infty);L^2(\Omega)\big)$ and
\begin{align*}
	&	u\in C_\b\big([0,\infty);H^1_0(\Omega)\big)\cap W^{1,\infty}\big((0,\infty);L^2(\Omega)\big)\cap L^\infty\big((0,\infty);L^{2p}(\Omega)\big),	\\
	&	u\in L^\infty\big((0,\infty);L^{2m}(\Omega)\big), \text{ if } m>0 \; \text{ and } \; U\in L^\infty\big((0,\infty);L^2(\Omega)\big), \text{ if } m=0,
\end{align*}
where $U$ is the saturated section associated to $u.$
\end{enumerate}
\end{thm}

\begin{rmk}
\label{rmkthmstrongH2}
Theorem~\ref{thmstrongH2} appeals to some comments.
\begin{enumerate}
\item
\label{rmkthmstrongH21}
Since $f\in W^{1,1}_\loc\big([0,\infty);L^2(\Omega)\big)\inj C\big([0,\infty);L^2(\Omega)\big),$ estimate~\eqref{strongH23} with $f(0)$ makes sense.
\item
\label{rmkthmstrongH22}
By H\"{o}lder's inequality, $\|u\|_{L^{p+1}(\Omega)}^{p+1}\le\|u\|_{L^{2p}(\Omega)}^p\|u\|_{L^2(\Omega)}$ and then $u\in C\big([0,\infty);L^{p+1}(\Omega)\big).$ Futhermore, if $v$ is also an $H^2$-solution then by H\"{o}lder's inequality,
\begin{gather*}
\||u|^{p-1}u-|v|^{p-1}v\|_{L^\frac{p+1}p(\Omega)}\le C\left(\|u\|_{L^{p+1}(\Omega)}^{p-1}+\|v\|_{L^{p+1}(\Omega)}^{p-1}\right)\|u-v\|_{L^{p+1}(\Omega)},
\end{gather*}
so that $|u|^{p-1}u\in C\big([0,\infty);L^\frac{p+1}p(\Omega)\big).$
\item
\label{rmkthmstrongH23}
In the same way, if $m>0$ then $u\in C\big([0,\infty);L^{m+1}(\Omega)\big)$ and $|u|^{m-1}u\in C\big([0,\infty);L^\frac{m+1}m(\Omega)\big)$ (\cite[Lemma~A.1]{MR4053613}). Therefore, $\|u(\,.\,)\|_{L^2(\Omega)}^2\in C^1\big([0,\infty);\R\big)$ and \eqref{L2} holds true for any $t\ge0.$ We then infer by the equation \eqref{gl} that
\begin{gather*}
\begin{cases}
\dfrac{\partial u}{\partial t}\in C_\w\big([0,\infty);L^2(\Omega)+L^\frac{m+1}m(\Omega)+L^\frac{p+1}p(\Omega)\big),	\medskip \\
u\in C^1\big([0,\infty);H^{-1}(\Omega)+L^\frac{m+1}m(\Omega)+L^\frac{p+1}p(\Omega)\big).
\end{cases}
\end{gather*}
\item
\label{rmkthmstrongH24}
Due to the regularity of $H^2$solutions, it follows that $u$ satisfies~\eqref{defsol11} with $q=m$ and $q=p.$
\end{enumerate}
\end{rmk}

\begin{rmk}
\label{rmkp}
All the results in this section remain valid without the presence of the superlinear term $b|u|^{p-1}u$, which amounts to taking $b=0.$ In this case, one must ignore all Lebesgue and Sobolev spaces that involve $p.$ In particular, if $u$ is a strong solution in the sense of the Definition~\ref{defsol}, then \eqref{defsol11} reads:
\begin{gather}
\label{rmkp1}
u\in L^{m+1}_\loc\big([0,\infty);H^1_0(\Omega)\cap L^{m+1}(\Omega)\big)\cap W^{1,\frac{m+1}m}_\loc\big([0,\infty);H^\star+L^\frac{m+1}m(\Omega)\big),
\end{gather}
as in the series of papers \cite{MR4098330,MR4053613,MR4340780,MR4503241,MR4725781}.
\end{rmk}

\section{Notations and some results of Functional Analysis}
\label{notfunana}

We collect here some notations that will be used along with this paper. Throughout this section, $X$ is a real Banach space, $\Omega$ is an arbitrary open subset of $\R^N$ and $I$ is an interval. The set of positive integers is denoted by $\N,$ and $\N_0=\N\cup\{0\}.$ Unless if specified, all functions are complex-valued and all the vector spaces are considered over the field $\R.$ For $1\le p\le\infty,$ $p^\prime$ is the conjugate of $p$ defined by $\frac{1}{p}+\frac{1}{p^\prime}=1.$ $X^\star$ is the topological dual of $X$ and $\langle\: . \; , \: . \:\rangle_{X^\star,X}\in\R$ is the $X^\star-X$ duality product. The scalar product in $L^2(\Omega)$ between two functions $u,v$ is, $(u,v)_{L^2(\Omega)}=\Re\int_{\Omega}u(x)\ovl{v(x)}\d x.$ For $p\in(0,\infty],$ $u\in L^p_\loc\big([0,\infty);X\big)$ means that $u\in L^p_\loc\big((0,\infty);X\big)$ and for any $T>0,$ $u_{|(0,T)}\in L^p\big((0,T);X\big).$ In the same way, we will use the notation $u\in W^{1,p}_\loc\big([0,\infty);X\big).$ If $p\in(0,\infty]$ and $r=0$ then $L^\frac{p}r(\Omega)=L^\infty(\Omega)$ and $W^{1,\frac{p}r}(\Omega)=W^{1,\infty}(\Omega).$ $L^0(\Omega)$ is the space of measurable functions $u:\Omega\tends\C$ such that $|u|<\infty,$ almost eveywhere in $\Omega.$ The Lebesgue measure of a measurable set $A\subset\R^N$ will be denoted by $|A|.$ $C_\w(I;X)$ is the space of continuous functions from $I$ to $\big(X,\sigma(X,X^\star)\big)(=X_\w).$ Finally, we denote by $C$ auxiliary positive constants, and sometimes, for positive parameters $a_1,\ldots,a_n,$ write as $C(a_1,\ldots,a_n)$ to indicate that the constant $C$ depends only and continuously on $a_1,\ldots,a_n$ (we will use this convention for constants which are not denoted merely by ``$C$'').
\medskip \\
Throughout this paper, we shall always identify $L^2(\Omega)$ with its topological dual. Let us recall some important results of Functional Analysis that we will use without necessarily referring to them. Let $E$ and $F$ be locally convex Hausdorff topological vector spaces. If $E\inj F$ with dense embedding then $F^\star\inj E^\star,$ and for any $L\in F^\star$ and $x\in E,$ $\langle L,x\rangle_{E^\star,E}=\langle L,x\rangle_{F^\star,F}.$ Moreover, if $E$ and $F$ are Banach spaces then so are $E\cap F$ and $E+F$ and if $E\cap F$ is dense in both $E$ and $F$ then $(E\cap F)^\star=E^\star+F^\star$ and $(E+F)^\star=E^\star\cap F^\star.$ If $\Dr(\Omega)\inj X$ with dense embedding then $L^1_\loc\big([0,\infty);X^\star\big)\inj\Dr^\p\big((0,\infty)\times\Omega\big).$ Let $p\in[1,\infty).$ If $X$ is separable then so is $L^p(I;X),$ and if $X$ is reflexive then $L^p(I;X)^\star\cong L^{p^\p}(I;X^\star).$ Let $X\inj Y$ be Banach spaces and $u\in C_\w(\ovl I;Y).$ Assume that there exist $C>0$ and a null set $N_0\subset I$ such that for any $t\in I\setminus N_0,$ $u(t)\in X$ and $\|u(t)\|_X\le C.$ If $X$ is reflexive then for any $t\in\ovl I,$ $u(t)\in X,$ $\|u(t)\|_X\le C$ and $u\in C_\w(\ovl I;X).$ Let $u\in\C^\Omega,$ $1<p\le\infty,$ and $(u_n)_{n\in\N}\subset L^p(\Omega)$ be bounded. If $u_n\tends u,$ as $n\to\infty,$ a.e.\;in $\Omega,$ then $u\in L^p(\Omega)$ and $u_n-\!\!\!\weak u,$ as $n\to\infty,$ in $L^p(\Omega)_\w,$ if $p<\infty,$ and in $L^\infty(\Omega)_{\w\star},$ if $p=\infty.$ Finally, we recall that an operator $(A,D(A))$ on $L^2(\Omega)$ is maximal monotone if, and only if, it is $m$-accretive. For more details, see Bergh and L\"{o}fstr\"{o}m~\cite{MR0482275}, Brezis~\cite{MR0348562}, Droniou~\cite{droniou}, Edwards~\cite{MR0221256}, Strauss~\cite{MR0205121,MR306715}, Tr\`{e}ves~\cite{MR2296978} and B\'{e}gout~\cite{MR4521439}.

\section{Maximal monotone operators}
\label{maxmon}

Throughout this paper, we shall use the following notations and conventions. Let $m\ge0.$ Since $\big||z|^{-(1-m)}z\big|=|z|^m,$ we extend by continuity at $z=0$ the map $z\longmapsto|z|^{-(1-m)}z$ by setting $|z|^{-(1-m)}z=0$ if $m>0$ and $z=0.$ To solve \eqref{gl}, we rewrite the equation as $\frac{\d u}{\d t}+A_0^mu=f,$ where $A_0^m$ is defined below, and use the theory of maximal operators in $L^2(\Omega).$ But the nonlinearities $|u|^{-(1-m)}u$ and $|u|^{p-1}u$ do not belong to $L^2(\Omega).$ To this end, we regularize the damping term around $u=0$ and truncate the superlinearity for large values of $u.$ As a consequence, we define the operators below. Let $\eps\ge0,$ $p>1$ and $M>0.$ For $u\in L^0(\Omega),$ we define
\begin{align*}
& g_\eps^m(u)=(|u|^2+\eps)^{-\frac{1-m}2}u,	\; m+\eps>0,		\\
& h_M^p(u)=|u|^{p-1}u\1_{\{|u|<M\}}+M^{p-1}u\1_{\{|u|\ge M\}},		\\
& g_0^0(u)=\frac{u}{|u|}, \; u\neq0.
\end{align*}
Obviously, for any $\eps\ge0,$ $M>0$ and $u\in L^0(\Omega),$ $g_\eps^1(u)=h_M^1(u)=u.$ In all this section, we suppose \eqref{O}, \eqref{t} and \eqref{g}. Let $a,b\in\C,$ $0\le m\le1,$ $p>1$ and $\eps\ge0.$ Let us define the following operators on $L^2(\Omega).$
\begin{align*}
&
\forall u\in D(L)\eqdef\left\{u\in H^1_0(\Omega); \; \Delta u\in L^2(\Omega)\right\}, \; Lu=-\e^{\vi\theta}\Delta u+\gamma\e^{\vi\theta}u,
\\
&
\begin{cases}
D(B_\eps^m)=L^2(\Omega), \; \eps>0 \text{ or } m=1,	\medskip \\
\forall u\in D(B_\eps^m), \; B_\eps^mu=a\e^{\vi\theta}g_\eps^m(u)+b\e^{\vi\theta}h_M^p(u),
\end{cases}
\begin{cases}
D(A_\eps^m)=D(L), \; \eps>0,	\medskip \\
\forall u\in D(A_\eps^m), \; A_\eps^mu=Lu+B_\eps^mu,
\end{cases}
\end{align*}
\begin{align*}
&
\begin{cases}
D(B_0^0)=L^2(\Omega),		\medskip \\
\forall u\in D(B_0^0), \; B_0^0u=\Big\{U\in L^2(\Omega); \; U \text{ is a saturated section associated to } u\Big\},
\end{cases}
\\[2pt]
&
\begin{cases}
D(A_0^0)=D(L)\cap L^{2p}(\Omega),				\medskip \\
\forall u\in D(A_0^0), \; A_0^0u=\big\{Lu+a\e^{\vi\theta}U+b\e^{\vi\theta}g_0^p(u); \; U\in B_0^0u\big\},
\end{cases}
\\[2pt]
&
\begin{cases}
D(A_0^m)=D(L)\cap L^{2m}(\Omega)\cap L^{2p}(\Omega), \; m>0,	\medskip \\
\forall u\in D(A_0^m), \; A_0^mu=Lu+a\e^{\vi\theta}g_0^m(u)+b\e^{\vi\theta}g_0^p(u).
\end{cases}
\end{align*}
We begin by establishing some preliminary lemmas on the monotonicity of the functions defined at the beginning of this section.

\begin{lem}
\label{lemge}
Let $0\le m\le1,$ $p\ge1,$ $\eps>0$ and $M>0.$ Then,
\begin{gather*}
g_\eps^m,h_M^p\in
C\big(L^2(\Omega);L^2(\Omega)\big)\cap C\big(H^1(\Omega);H^1(\Omega)\big)\cap C\big(H^1_0(\Omega);H^1_0(\Omega)\big).
\end{gather*}
In addition, for any $u\in H^1(\Omega),$
\begin{gather}
\label{lemge1}
\nabla g_\eps^m(u)=(m-1)\Re(\ovl u\nabla u)(|u|^2+\eps)^\frac{m-3}2u+(|u|^2+\eps)^\frac{m-1}2\nabla u,	\\
\label{lemge2}
\nabla h_M^p(u)=\nabla g_0^p(u)\1_{\{|u|<M\}}+M^{p-1}\nabla u\1_{\{|u|\ge M\}},
\end{gather}
almost everywhere in $\Omega,$ where $\nabla g_0^p(u)$ in \eqref{lemge2} is given by \eqref{lemge1} with $m=p$ and $\eps=0.$
\end{lem}

\begin{rmk}
\label{rmklemge2}
Let $p$ and $M$ be as in the lemma. For $z\in\C,$ we set
\begin{gather*}
h_M^p(z)=|z|^{p-1}z\1_{\{|z|<M\}}+M^{p-1}z\1_{\{|z|\ge M\}}.
\end{gather*}
Since $h_M^p$ is globally Lipschitz continuous, we would like to prove \eqref{lemge2} with the help of the theory of Nemytskii superposition operators. But the set on which $h_M^p,$ seen as a function from $\R^2$ to $\R^2,$ is not $C^1$ is infinite, since it is ${\{|z|=M\}}.$ Therefore, the formula
\begin{gather*}
\nabla h_M^p(u)=Dh_M^p(u)\nabla u,
\end{gather*}
is not valid and we have to proceed in another way.
\end{rmk}

\begin{vproof}{of Lemma~\ref{lemge}.}
Let $m,p,\eps$ and $M$ be as in the lemma. A straightforward calculation shows that for any $u,v\in L^0(\Omega),$
\begin{gather*}
|g_\eps^m(u)-g_\eps^m(v)|+|h_M^p(u)-h_M^p(v)|\le C(M,p,m,\eps)|u-v|,
\end{gather*}
so that $g_\eps^m,h_M^p\in C\big(L^2(\Omega);L^2(\Omega)\big).$ Let $u\in H^1(\Omega).$ By Meyer-Serrin's Theorem and the partial converse of the dominated convergence Theorem, there exist $(u_n)_{n\in\N}\subset H^1(\Omega)\cap C^\infty(\Omega)$ and $\Gamma\in L^2(\Omega;\R)$ such that
\begin{gather}
\label{demlemge1}
u_n\xrightarrow[n\to\infty]{H^1(\Omega)}u, \; u_n\xrightarrow[n\to\infty]{\text{a.e.\:in }\Omega}u, \; \nabla u_n\xrightarrow[n\to\infty]{\text{a.e.\:in }\Omega}\nabla u \text{ and for any } n\in\N, \;
|\nabla u_n|\overset{\text{a.e.}}{\le}\Gamma.
\end{gather}
Let $n\in\N.$ A straightforward calculation gives that $\nabla g_\eps^m(u_n),\nabla h_M^p(u_n)\in L^2(\Omega)$ and
\begin{gather}
\label{demlemge2}
\nabla g_\eps^m(u_n)=(m-1)\Re(\ovl u_n\nabla u_n)(|u_n|^2+\eps)^\frac{m-3}2u_n+(|u_n|^2+\eps)^\frac{m-1}2\nabla u_n,		\\
\label{demlemge3}
\nabla h_M^p(u_n)=
\begin{cases}
(p-1)\Re(\ovl u_n\nabla u_n)|u_n|^{p-3}u_n+|u_n|^{p-1}\nabla u_n,	&	\text{if } |u_n|<M,	\medskip \\
M^{p-1}\nabla u_n,										&	\text{if } |u_n|>M.
\end{cases}
\end{gather}
We use \eqref{demlemge1}, the dominated convergence Theorem, and we let $n\nearrow\infty$ in \eqref{demlemge2}. We then obtain that $g_\eps^m:H^1(\Omega)\tends H^1(\Omega)$ and \eqref{lemge1}. Now, let us recall that if $v\in H^1(\Omega)$ then $|v|\in H^1(\Omega;\R),$ $\nabla|v|=0,$ a.e.\:in $\{|v|=M\},$ and thus $\Re(\ovl v\nabla v)=|v|\nabla|v|=0,$ a.e.\:in $\{|v|=M\}$ (see, for instance, Theorem~6.17, p.152, in~Lieb and Loss~\cite{MR1817225}). It follows that,
\begin{gather}
\label{demlemge4}
\nabla h_M^p(u_n)=\nabla g_0^p(u_n)\1_{\{|u_n|<M\}}+M^{p-1}\nabla u_n\1_{\{|u_n|\ge M\}},
\end{gather}
Using \eqref{demlemge1}, \eqref{demlemge3} and the dominated convergence Theorem, we may let $n\nearrow\infty$ in \eqref{demlemge4}. We get that $h_M^p(u)\in H^1(\Omega)$ and \eqref{lemge2}. With the same arguments, we show that $g_\eps^m,h_M^p\in C\big(H^1(\Omega);H^1(\Omega)\big).$ Now, let us show that $h_M^p\in C\big(H^1_0(\Omega);H^1_0(\Omega)\big).$ Let $u\in H^1_0(\Omega).$ It is enough to swow that $h_M^p(u)\in H^1_0(\Omega).$ Let $(\vphi_n)_{n\in\N}\subset\Dr(\Omega)$ be such that $\vphi_n\tends u,$ as $n\to\infty,$ in $H^1_0(\Omega).$ Since $(h_M^p(\vphi_n))_{n\in\N}\subset H^1(\Omega)$ and $\supp h_M^p(\vphi_n)$ is compact for any $n\in\N,$ it follows that $(h_M^p(\vphi_n))_{n\in\N}\subset H^1_0(\Omega)$ and, by continuity, that $h_M^p(u)\in H^1_0(\Omega)$ and $h_M^p(\vphi_n)\tends h_M^p(u),$ as $n\to\infty,$ in $H^1_0(\Omega).$ We show in the same way that $g_\eps^m\in C\big(H^1_0(\Omega);H^1_0(\Omega)\big).$ This ends the proof of the lemma.
\medskip
\end{vproof}

\begin{lem}
\label{lemeneregint}
Let $q,\eps\ge0$ and let $\Omega^\p\subset\R^N$ be measurable $($not necessarily open$).$ Assume that there exists $A:H^1(\Omega^\p)\tends H^1(\Omega^\p)$ such that for any $u\in H^1(\Omega^\p),$
\begin{gather}
\label{demlemeneregint1}
\nabla A(u)=(q-1)\Re(\ovl u\nabla u)(|u|^2+\eps)^\frac{q-3}2u+(|u|^2+\eps)^\frac{q-1}2\nabla u,
\end{gather}
almost everywhere in $\Omega^\p,$ where both terms belong to $L^2(\Omega^\p).$ Then,
\begin{gather}
\label{demlemeneregint2}
|1-q|\Re\left(\,\vint_{\Omega^\p}\nabla u.\ovl{\nabla A(u)}\d x\right)\ge2\sqrt{q}\,\left|\Im\left(\,\vint_{\Omega^\p}\nabla u.\ovl{\nabla A(u)}\d x\right)\right|,
\end{gather}
for any $u\in H^1(\Omega^\p).$
\end{lem}

\begin{vproof}{of Lemma~\ref{lemeneregint}.}
Let the assumptions of the lemma be fulfilled. We have by \eqref{demlemeneregint1} that,
\begin{gather*}
I\eqdef\vint_{\Omega^\p}\nabla u.\ovl{\nabla A(u)}\d x
=\vint_{\Omega^\p}\frac{|\nabla u|^2(|u|^2+\eps)-(1-q)\Re(\ovl u\nabla u).\ovl u\nabla u}{(|u|^2+\eps)^\frac{3-q}2}\d x,
\end{gather*}
Using that $|\nabla u|^2|u|^2=|\Re(\ovl u\nabla u)|^2+|\Im(\ovl u\nabla u)|^2,$ we get that
\begin{align*}
&	\Re(I)=\eps\,\vint_{\Omega^\p}\frac{|\nabla u|^2}{(|u|^2+\eps)^\frac{3-q}2}\d x
			+\vint_{\Omega^\p}\frac{q|\Re(\ovl u\nabla u)|^2+|\Im(\ovl u\nabla u)|^2}{(|u|^2+\eps)^\frac{3-q}2}\d x,	\\
&	\Im(I)=(1-q)\vint_{\Omega^\p}\frac{\Re(\ovl u\nabla u).\Im(\ovl u\nabla u)}{(|u|^2+\eps)^\frac{3-q}2}\d x.
\end{align*}
In particular, this gives the result if $q=0,$ so that we now assume that $q>0.$ By Young's inequality,
\begin{gather*}
2|\Re(\ovl u\nabla u)|\,|\Im(\ovl u\nabla u)|\le\sqrt q\,|\Re(\ovl u\nabla u)|^2+\frac{|\Im(\ovl u\nabla u)|^2}{\sqrt q}.
\end{gather*}
It follows that,
\begin{align*}
	&	\; 2\sqrt{q}\,|\Im(I)|\le|1-q|\vint_{\Omega^\p}\frac{2\sqrt{q}\,|\Re(\ovl u\nabla u)|\,|\Im(\ovl u\nabla u)|}{(|u|^2+\eps)^\frac{3-q}2}\d x	\\
  \le	&	\; |1-q|\vint_{\Omega^\p}\frac{q|\Re(\ovl u\nabla u)|^2+|\Im(\ovl u\nabla u)|^2}{(|u|^2+\eps)^\frac{3-q}2}\d x\le|1-q|\Re(I),
\end{align*}
which gives \eqref{demlemeneregint2}.
\medskip
\end{vproof}

\noindent
The lemma below concerns the accretivity of the operators defined at the beginning of this section. A part of these results (estimate \eqref{lemenereg21} below) can be found in Pazy~\cite{MR710486} (proof of Theorem~3.6, p.215--216; see also Okazawa~\cite{MR1072347}), but for the sake of completeness, we give the full proof.

\begin{lem}
\label{lemenereg}
Let $m\ge0,$ $\eps>0,$ $M>0,$ $p>1$ and $u\in D(L).$
\begin{enumerate}
\item
\label{lemenereg1}
Suppose $m=0.$ Let $U\in L^1_\loc(\Omega)$ be such that $U=\dfrac{u}{|u|},$ a.e.\,in $\{u\neq0\}.$ If $U\Delta u\in L^1(\Omega)$ then
\begin{gather}
\label{lemenereg11}
\Re\left(\,\dsp\vint_\Omega(-\ovl{\Delta u})U\d x\right)\ge0.
\end{gather}
In particular, if $a\in C_\theta(0)$ then $\Re\left(a\e^{\vi\theta}\dsp\vint_\Omega(-\ovl{\Delta u})U\d x\right)\ge0.$
\item
\label{lemenereg2}
Assume that $m>0.$ If $u^m\Delta u\in L^1(\Omega)$ then we have
\begin{gather}
\label{lemenereg21}
|1-m|\Re\left(\,\vint_\Omega(-\Delta u)\ovl{g_0^m(u)}\d x\right)\ge2\sqrt{m}\,\left|\Im\left(\,\vint_\Omega(-\Delta u)\ovl{g_0^m(u)}\d x\right)\right|.
\end{gather}
If $a\in C_\theta(m)$ then $\Re\left(a\e^{\vi\theta}\dsp\vint_\Omega(-\ovl{\Delta u})g_0^m(u)\d x\right)\ge0.$
\item
\label{lemenereg3}
Assume that $m\in[0,1].$ We have,
\begin{gather}
\label{lemenereg31}
(1-m)\Re\left(\,\vint_\Omega(-\Delta u)\ovl{g_\eps^m(u)}\d x\right)\ge2\sqrt{m}\,\left|\Im\left(\,\vint_\Omega(-\Delta u)\ovl{g_\eps^m(u)}\d x\right)\right|.
\end{gather}
If $a\in C_\theta(m)$ then $\Re\left(a\e^{\vi\theta}\dsp\vint_\Omega(-\ovl{\Delta u})g_\eps^m(u)\d x\right)\ge0.$
\item
\label{lemenereg4}
We have,
\begin{gather}
\label{lemenereg41}
(p-1)\Re\left(\,\vint_\Omega(-\Delta u)\ovl{h_M^p(u)}\d x\right)\ge2\sqrt{p}\,\left|\Im\left(\,\vint_\Omega(-\Delta u)\ovl{h_M^p(u)}\d x\right)\right|.
\end{gather}
If $b\in C_\theta(p)$ then $\Re\left(b\e^{\vi\theta}\dsp\vint_\Omega(-\ovl{\Delta u})h_M^p(u)\d x\right)\ge0.$
\end{enumerate}
\end{lem}

\begin{proof*}
We first establish \eqref{lemenereg31} (and so $m\le1).$ By Lemma~\ref{lemge}, we may apply Lemma~\ref{lemeneregint} with $q=m,$ $\Omega^\p=\Omega$ and $A=g_\eps^m,$ from which \eqref{lemenereg31} follows since
\begin{gather*}
\vint_\Omega(-\Delta u)\ovl{g_\eps^m(u)}\d x=\vint_\Omega\nabla u.\ovl{\nabla g_\eps^m(u)}\d x.
\end{gather*}
Letting $\eps\searrow0$ in \eqref{lemenereg31} with the help of the dominated convergence Theorem, we obtain \eqref{lemenereg21} for $0<m\le1.$ We continue with the proof of \eqref{lemenereg41}. Still by Lemma~\ref{lemge}, we may apply Lemma~\ref{lemeneregint} with $q=p,$ $\Omega^\p=\Omega_M=\left\{|u|<M\right\}$ and $A=h_M^p.$ It follows that
\begin{align*}
	&	\; (p-1)\Re\left(\int_\Omega\nabla u.\ovl{\nabla h_M^p(u)}\d x\right)
			=(p-1)\Re\left(\int_{\Omega_M}\nabla u.\ovl{\nabla g_0^p(u)}\d x+M^{p-1}\int_{\Omega_M^\co}|\nabla u|^2\d x\right)	\\
  \ge	&	\; 2\sqrt{p}\,\Im\left(\int_{\Omega_M}\nabla u.\ovl{\nabla g_0^p(u)}\d x+M^{p-1}\int_{\Omega_M^\co}|\nabla u|^2\d x\right)
			=2\sqrt{p}\,\Im\left(\int_\Omega\nabla u\ovl{\nabla h_M^p(u)}\d x\right),
\end{align*}
from which \eqref{lemenereg41} follows since by Lemma~\ref{lemge},
\begin{gather*}
\vint_\Omega(-\Delta u)\ovl{h_M^p(u)}\d x=\vint_\Omega\nabla u.\ovl{\nabla h_M^p(u)}\d x.
\end{gather*}
With the help of the dominated convergence Theorem, we let $M\nearrow\infty$ to the above and we get \eqref{lemenereg21} for $m>1.$ Now we turn out to the proof of the last inequalities of Properties~\ref{lemenereg2}--\ref{lemenereg4}. Let $a\in C_\theta(m)$ and let $I$ be the integral in \eqref{lemenereg21}. We have by \eqref{lemenereg21} that,
\begin{gather*}
2\sqrt m\,\Re(a\e^{\vi\theta} I)=2\sqrt m\,\Re(a\e^{\vi\theta})\Re(I)+2\sqrt m\,\Im(a\e^{\vi\theta})\Im(\ovl I)\ge0,
\end{gather*}
which is the last inequality of Property~\ref{lemenereg2}. The last inequalities of Properties~\ref{lemenereg3}--\ref{lemenereg4} are obtained in the same way. To conclude the proof of the lemma, it remains to establish Property~\ref{lemenereg1}. Assume that $m=0$ and let $U$ be as in the lemma. Set $\omega=\big\{u\neq0\big\}.$ Since $u\in D(L)$ then $u\in H^2_\loc(\Omega)$ and it follows that $\Delta u=0,$ a.e.\,in $\omega^\co.$ 
It follows that the integrals in \eqref{lemenereg11} and \eqref{lemenereg31} are the same if we replace $\Omega$ with $\omega.$ Then, applying \eqref{lemenereg31} applied with $m=0,$ and using Lebesgue's Theorem, we let $\eps\searrow0$ in \eqref{lemenereg31} and obtain \eqref{lemenereg11}, while the last inequality is obvious.
\medskip
\end{proof*}

\begin{lem}
\label{lemL}
The operator $(L,D(L))$ is maximal monotone on $L^2(\Omega).$
\end{lem}

\begin{proof*}
By \eqref{t}, \eqref{g} and the Lax-Milgram Theorem, $L$ is monotone and $R(I+L)=L^2(\Omega).$
\medskip
\end{proof*}

\begin{lem}
\label{lemg}
Let $\big(0\le m\le1$ and $\eps>0\big)$ or $\big(m>0$ and $\eps=0\big),$ and let $u,v\in L^1_\loc(\Omega)$ be such that $\big(g_\eps^m(u)-g_\eps^m(v)\big)\big(\ovl{u-v}\big)\in L^1(\Omega).$ If $a\in C_\theta(m)$ then we have,
\begin{gather*}
\Re\left(a\e^{\vi\theta}\vint_\Omega\big(g_\eps^m(u)-g_\eps^m(v)\big)(\ovl{u-v})\d x\right)\ge0.
\end{gather*}
\end{lem}

\begin{proof*}
Let $u,v\in L^1_\loc(\Omega)$ be such that $\big(g_\eps^m(u)-g_\eps^m(v)\big)\big(\ovl{u-v}\big)\in L^1(\Omega),$ and let $I$ be the integral in the lemma. We have by Liskevich and Perel$^\p$muter~\cite[Lemma~2.2]{MR1224619} $(m>0$ and $\eps=0),$ \cite[Lemma~5.6]{MR4340780} $(m\in[0,1]$ and $\eps>0)$ and \eqref{a} that $2\sqrt m\,\Re(a\e^{\vi\theta}I)=2\sqrt m\left(\Re(a\e^{\vi\theta})\Re(I)-\Im(a\e^{\vi\theta})\Im(I)\right)$ and
\begin{gather*}
2\sqrt m\left(\Re(a\e^{\vi\theta})\Re(I)-\Im(a\e^{\vi\theta})\Im(I)\right)\ge\left(2\sqrt m\,\Re(a\e^{\vi\theta})-|1-m|\,|\Im(a\e^{\vi\theta})|\right)\Re(I)\ge0.
\end{gather*}
Hence the result.
\medskip
\end{proof*}

\begin{lem}
\label{lemh}
Let $p>1$ and $M>0.$ Let for $z\in\C,$ $\phi_M(z)=|z|^{p-1}z\1_{\{|z|<M\}}+M^{p-1}z\1_{\{|z|\ge M\}}.$ Then for any $(z_1,z_2)\in\C^2,$
\begin{gather}
\label{lemh1}
(p-1)\Re\left(\big(\phi_M(z_1)-\phi_M(z_2)\big)\big(\ovl{z_1-z_2}\big)\right)
\ge2\sqrt p\left|\Im\left(\big(\phi_M(z_1)-\phi_M(z_2)\big)\big(\ovl{z_1-z_2}\big)\right)\right|.
\end{gather}
In addition, if $b\in C_\theta(p)$ then for any $u,v\in L^1_\loc(\Omega)$ such that $\big(h_M^p(u)-h_M^p(v)\big)(u-v)\in L^1(\Omega),$
\begin{gather}
\label{lemh2}
\Re\left(b\e^{\vi\theta}\vint_\Omega\big(h_M^p(u)-h_M^p(v)\big)(\ovl{u-v})\d x\right)\ge0,
\end{gather}
\end{lem}

\begin{proof*}
Let the assumption of the lemma be fulfilled. We begin by proving \eqref{lemh1}. This result is obvious if $z_1z_2=0$ or if $|z_1|,|z_2|\ge M.$ If $|z_1|,|z_2|<M$ then \eqref{lemh1} due to Liskevich and Perel$^\p$muter~\cite[Lemma~2.2]{MR1224619}. So, without loss of generality, we may assume that $0<|z_2|<M\le|z_1|.$ Let,
\begin{gather*}
Z_M=\big(\phi_M(z_1)-\phi_M(z_2)\big)\big(\ovl{z_1-z_2}\big) \; \text{ and } \; F_M(t,s,\theta)=\frac{|\Im(Z_M)|}{\Re(Z_M)}.
\end{gather*}
(By \cite[Lemma~5.3]{MR4340780}, $\Re(Z_M)>0.)$ Setting $t=|z_1|,$ $s=|z_2|$ and $\theta=\Arg(\ovl{z_1}z_2),$ we have to show that,
\begin{gather}
\label{demlemh1}
F_M(t,s,\theta)^2\le\frac{(p-1)^2}{4p}.
\end{gather}
We have that,
\begin{gather*}
F_M(t,s,\theta)^2=\frac{t^2s^2((M^{p-1}-s^{p-1})^2(1-\cos^2\theta)}{\big((M^{p-1}t^2+s^{p+1})-ts(M^{p-1}+s^{p-1})\cos\theta\big)^2}
\eqdef\frac{A(1-\cos^2\theta)}{(B-C\cos\theta)^2},
\end{gather*}
where we have used that $\Re(z_1\ovl{z_2})=ts\cos\theta$ and $\Im(\ovl{z_1}z_2)=ts\sin\theta.$ We claim that
\begin{gather}
\label{demlemh2}
F_M(t,s,\theta)^2\le\dfrac{t^2s^2(M^{p-1}-s^{p-1})^2}{(t^2-s^2)(M^{2p-2}t^2-s^{2p})}.
\end{gather}
We write $\sigma=\cos\theta$ and $g(\sigma)=F_M(t,s,\theta)^2.$ Note that since $t\ge M>s>0$ and $\Re(Z_M)>0,$ we have $A>0,$ and $B-C\sigma>0,$ for any $\sigma\in[-1,1].$ Since, $\vmax_{\sigma\in[-1,1]}g(\sigma)=g\left(\frac{C}{B}\right)=\frac{A}{B^2-C^2},$ it follows that, $\vsup_{\theta\in\R}F_\eps(t,s,\theta)^2\le\frac{A}{B^2-C^2},$ which gives \eqref{demlemh2}. Moreover, by Cauchy-Schwarz' inequality, we have,
\begin{align*}
	&	\; \frac{(M^{p-1}-s^{p-1})^2}{(p-1)^2}=\left(\:\vint_s^M\sigma^{p-2}\d\sigma\right)^2
			=\left(\:\vint_s^M\sigma^{p-\frac12}\sigma^{-\frac32}\d\sigma\right)^2							\\
  \le	&	\; \vint_s^M\sigma^{2p-1}\d\sigma\vint_s^t\sigma^{-3}\d\sigma=\frac{(M^{2p}-s^{2p})(s^{-2}-t^{-2})}{4p}
			\le\frac{t^2-s^2}{4\,p\,t^2s^2}(M^{2p-2}t^2-s^{2p}).
\end{align*}
Putting together \eqref{demlemh2} with above, we get \eqref{demlemh1} and then \eqref{lemh1}. Now we turn out to the proof of \eqref{lemh2} by using \eqref{lemh1}. Let $(z_1,z_2)\in\C^2.$ We have that,
\begin{align*}
	& \; 2\sqrt p\,\Re\left(b\e^{\vi\theta}\big(\phi_M(z_1)-\phi_M(z_2)\big)\big(\ovl{z_1-z_2}\big)\right)		\\
   =	& \; 2\sqrt p\left(\Re(b\e^{\vi\theta})\Re\left(\big(\phi_M(z_1)-\phi_M(z_2)\big)\big(\ovl{z_1-z_2}\big)\right)
   		-\Im(b\e^{\vi\theta})\Re\left(\big(\phi_M(z_1)-\phi_M(z_2)\big)\big(\ovl{z_1-z_2}\big)\right)\right)		\\
  \ge	& \; \left(2\sqrt p\,\Re(b\e^{\vi\theta})-(p-1)|\Im(b\e^{\vi\theta})|\right)
		\Re\left(\big(\phi_M(z_1)-\phi_M(z_2)\big)\big(\ovl{z_1-z_2}\big)\right)\ge0,
 \end{align*}
from which we get \eqref{lemh2}.
\medskip
\end{proof*}

\noindent
Now, Let us prove the monotonicity of the operators introduced at the beginning of this section.

\begin{lem}
\label{lemAmon}
Let $(m,p)\in[0,1]\times(1,\infty)$ and $(a,b)\in C_\theta(m)\times C_\theta(p).$ Then $(A_0^m,D(A_0^m))$ is monotone.
\end{lem}

\begin{proof*}
If $m>0$ then the result comes from Lemmas~\ref{lemL} and \ref{lemg}. Now, assume that $m=0.$ Since $a\in C_\theta(0),$ we have by \eqref{C0'} that $a=\mu\e^{-\vi\theta},$ for some positive real number $\mu.$ Let $u_1,u_2\in D(A_0^0)$ and $(V_1,V_2)\in A_0^0u_1\times A_0^0u_2.$ Then for each $j\in\{1,2\},$ there exists $U_j\in B_0^0u_j$ such that $V_j=L u_j+\mu U_j+b\e^{\vi\theta}g_0^p(u_j).$ By accretiveness of the operator $L$ (Lemma~\ref{lemL}) and Lemma~\ref{lemg},
\begin{gather*}
(V_1-V_2,u_1-u_2)_{L^2(\Omega)}\ge\mu(U_1-U_2,u_1-u_2)_{L^2(\Omega)}.
\end{gather*}
By \cite[Lemma~6.1]{MR4725781}, we have for any $n\in\N$ that $\Re\left(\int_{\Omega\cap B(0,n)}(U_1-U_2)\big(\ovl{u_1-u_2}\big)\d x\right)\ge0.$ The result follows from Lebesgue's Theorem.
\medskip
\end{proof*}

\begin{prop}
\label{propAe}
Let $M>
0,$ $(m,p)\in[0,1]\times(1,\infty),$ $(a,b)\in C_\theta(m)\times C_\theta(p)$ and $\eps>0.$ Then the operators $(A_\eps^m,D(A_\eps^m)),$ with $m<1,$ and $(A_0^1,D(A_0^1))$ are maximal monotone on $L^2(\Omega).$
\end{prop}

\begin{proof*}
Let $u,v\in L^2(\Omega).$ By Lemmas~\ref{lemg} and \ref{lemh}, the operators $(B_\eps^m,L^2(\Omega))$ and $(B_0^1,L^2(\Omega))$ are monotone. In addition, $B_\eps^m,B_0^1\in C\big(L^2(\Omega);L^2(\Omega)\big)$ (Lemma~\ref{lemge}) and we then deduce that $(B_\eps^m,L^2(\Omega))$ and $(B_0^1,L^2(\Omega))$ are maximal monotone (Brezis~\cite[Corollary~2.5, p.33]{MR0348562}). With the help of Lemma~\ref{lemL} and Brezis~\cite[Corollary~2.7, p.36]{MR0348562}, we obtain the desired result.
\medskip
\end{proof*}

\begin{prop}
\label{propAmax}
Under Assumption~$\ref{ass},$ the operator $(A_0^m,D(A_0^m))$ is maximal monotone.
\end{prop}

\begin{proof*}
By Proposition~\ref{propAe}, we may assume that $m\in[0,1).$ In addition, by Lemma~\ref{lemAmon}, we only have to show that $R(I+A_0^m)=L^2(\Omega).$ Let $F\in L^2(\Omega).$ We proceed with the proof in four steps. \\
\textbf{Step 1:} Let $\eps>0.$ There exists $u_\eps\in D(A_\eps^m)$ satisfying,
\begin{gather}
\label{lemAmax1}
-\e^{\vi\theta}\Delta u_\eps+a\e^{\vi\theta}g_\eps^m(u_\eps)+b\e^{\vi\theta}h_{\eps^{-1}}^p(u_\eps)+\gamma\e^{\vi\theta}u_\eps+u_\eps=F, \text{ in } L^2(\Omega).
\end{gather}
Since $(A_\eps^m,D(A_\eps^m))$ is maximal monotone (Proposition~\ref{propAe}), we have $R(I+A_\eps^m)=L^2(\Omega).$
\\
\textbf{Step 2:} The family $(u_\eps)_{\eps>0}$ is bounded in $H^1_0(\Omega),$ and the families $(\Delta u_\eps)_{\eps>0},$ $(g_\eps^m(u_\eps))_{\eps>0}$ and $(h_{\eps^{-1}}(u_\eps))_{\eps>0}$ are bounded in $L^2(\Omega).$ In addition, there exist a $u\in D(A_0^m)$ and a decreasing sequence $(\eps_n)_{n\in\N}\subset(0,1)$ converging toward $0$ such that,
\begin{align}
\label{lemAmax2}
&	u_{\eps_n}\xrightarrow[n\to\infty]{\Dr^\p(\Omega)}u,								\\
\label{lemAmax3}
&	u_{\eps_n}\xrightarrow[n\to\infty]{\text{a.e.\,in }\Omega}u,							\\
\label{lemAmax4}
&	g_{\eps_n}^m(u_{\eps_n})\xrightarrow[n\to\infty]{\Dr^\p(\Omega)}g_0^m(u), \; m>0,		\\
\label{lemAmax5}
&	h_{\eps_n^{-1}}^p(u_{\eps_n})\xrightarrow[n\to\infty]{\Dr^\p(\Omega)}g_0^p(u).
\end{align}
Let $\eps>0$ and $\Omega_\eps=\left\{|u_\eps|\le\eps^{-1}\right\}.$ We take the $L^2$-scalar product of~\eqref{lemAmax1} with $u_\eps.$ We get,
\begin{gather}
\begin{split}
\label{lemAmax6}
\cos\theta\|\nabla u_\eps\|_{L^2(\Omega)}^2+\Re(a\e^{\vi\theta})\vint_\Omega\frac{|u_\eps|^2}{(|u_\eps|^2+\eps)^\frac{1-m}2}\d x
				&	+\Re(b\e^{\vi\theta})\vint_\Omega h_{\eps^{-1}}^p(u_\eps)\ovl{u_\eps}\d x									\\
			+	&	\left(\Re(\gamma\e^{\vi\theta})+1\right)\|u_\eps\|_{L^2(\Omega)}^2=\Re\vint_\Omega F\,\ovl{u_\eps}\d x,
\end{split}
\end{gather}
Applying Cauchy-Schwarz' and Young's inequalities to~\eqref{lemAmax6}, we obtain that
\begin{gather}
\label{lemAmax7}
\begin{split}
\cos\theta\|u_\eps\|_{H^1_0(\Omega)}^2&+\Re(a\e^{\vi\theta})\vint_\Omega\frac{|u_\eps|^2}{(|u_\eps|^2+\eps)^\frac{1-m}2}\d x		\\
&+\Re(b\e^{\vi\theta})\left(\|u_\eps\|_{L^{p+1}(\Omega_\eps)}^{p+1}+\eps^{-(p-1)}\|u_\eps\|_{L^2(\Omega_\eps^\co)}^2\right)\le\|F\|_{L^2(\Omega)}^2.
\end{split}
\end{gather}
Repeating the process with $-\Delta u_\eps,$ we get with the help of Lemma \ref{lemenereg} that
\begin{gather}
\label{lemAmax8}
\cos\theta\|\Delta u_\eps\|_{L^2(\Omega)}^2\le\|F\|_{L^2(\Omega)}^2.
\end{gather}
Taking the $L^2$-scalar product of~\eqref{lemAmax1} with $g_\eps^m(u_\eps),$ we obtain,
\begin{gather}
\begin{split}
\label{lemAmax9}
\Re\left(\e^{\vi\theta}\vint_\Omega(-\Delta u_\eps)\ovl{g_\eps^m(u_\eps)}\d x\right)+&\;\Re(a\e^{\vi\theta})\|g_\eps^m(u_\eps)\|_{L^2(\Omega)}^2
			+\Re\left(b\e^{\vi\theta}\vint_\Omega h_{\eps^{-1}}^p(u_\eps)\ovl{g_\eps^m(u_\eps)}\d x\right)							\\
+\big(&\Re(\gamma\e^{\vi\theta})+1\big)\vint_\Omega\frac{|u_\eps|^2}{(|u_\eps|^2+\eps)^\frac{1-m}2}\d x
			=\Re\left(\,\vint_\Omega F\,\ovl{g_\eps^m(u_\eps)}\d x\right).
\end{split}
\end{gather}
Since $\Re(b\e^{\vi\theta})\ge0$ and
\begin{gather*}
\vint_\Omega h_{\eps^{-1}}^p(u_\eps)\ovl{g_\eps^m(u_\eps)}\d x
=\vint_{\Omega_\eps}\frac{|u_\eps|^{p+1}}{(|u_\eps|^2+\eps)^\frac{1-m}2}\d x
+\eps^{-(p-1)}\vint_{\Omega_\eps^\co}\frac{|u_\eps|^2}{(|u_\eps|^2+\eps)^\frac{1-m}2}\d x\ge0,
\end{gather*}
we get from \eqref{lemAmax8}, \eqref{lemAmax9} and the Cauchy-Schwarz and Young inequalities that
\begin{gather}
\label{lemAmax10}
\sup_{\eps>0}\|g_\eps^m(u_\eps)\|_{L^2(\Omega)}<\infty.
\end{gather}
It then follows from the equation \eqref{lemAmax1}, \eqref{lemAmax7}, \eqref{lemAmax8} and \eqref{lemAmax10} that,
\begin{gather}
\label{lemAmax11}
\sup_{\eps>0}\|h_{\eps^{-1}}^p(u_\eps)\|_{L^2(\Omega)}<\infty.
\end{gather}
By \eqref{lemAmax7} and \eqref{lemAmax8}, there exist $u\in D(L)$ and a decreasing sequence $\eps_n\searrow0$ satisfying \eqref{lemAmax2}--\eqref{lemAmax3}. By \eqref{lemAmax3}, we have that $h_{\eps_n^{-1}}^p(u_{\eps_n})\xrightarrow[n\to\infty]{\text{a.e.\,in }\Omega}g_0^p(u),$ which gives with \eqref{lemAmax11} that $u\in L^{2p}(\Omega)$ and \eqref{lemAmax5}. Note that if $m=0,$ we have shown that $u\in D(L)\cap L^{2p}(\Omega)=D(A_0^0)$ and Step~2 is ended. Assume $m>0.$ By \eqref{lemAmax3}, we have that
\begin{gather}
\label{lemAmax12}
g_{\eps_n}^m(u_{\eps_n})\xrightarrow[n\to\infty]{\text{a.e.\,in }\Omega}g_0^m(u).
\end{gather}
By \eqref{lemAmax10} and \eqref{lemAmax12}, we obtain that $u\in L^{2m}(\Omega)$ and \eqref{lemAmax4}, so that $u\in D(A_0^m).$
\\
\textbf{Step 3:} If $m=0$ then there exists $U\in B_0^0u$ such that, renumbering the sequence if necessary, $g_{\eps_n}^0(u_{\eps_n})\xrightarrow[n\to\infty]{\Dr^\p(\Omega)}U.$
\\
We have for any $n\in\N,$ $\|g_{\eps_n}^0(u_{\eps_n})\|_{L^\infty(\Omega)}\le1.$ Then, renumbering the sequence if necessary, there exists $U\in L^\infty(\Omega)$ such that $\|U\|_{L^\infty(\Omega)}\le1$ and $g_{\eps_n}^0(u_{\eps_n})-\!\!\!\!\weak U,$ as $n\to\infty,$ in $L^\infty(\Omega)_{\w\star}.$ Hence, $g_{\eps_n}^0(u_{\eps_n})\tends U,$ as $n\to\infty,$ in $\Dr^\p(\Omega).$ By \eqref{lemAmax3}, $g_{\eps_n}^0(u_{\eps_n})\tends g_0^0(u),$ as $n\to\infty,$ a.e.\,where $u\neq0.$ It follows from \cite[Lemma~6.1]{MR4848642} that $U=g_0^0(u),$ a.e.\,where $u\neq0.$ Finally by Step~2, $(g_\eps^0(u_\eps))_{\eps>0}$ is bounded in $L^2(\Omega).$ It follows that $U\in L^2(\Omega)$ and $g_{\eps_n}^0(u_{\eps_n})\underset{n\to\infty}{-\!\!\!-\!\!\!-\!\!\!-\!\!\!\weak}U,$ in $L^2(\Omega)_\w.$ Hence, $U\in B_0^0u.$
\\
\textbf{Step~4:} Conclusion.
\\
By~\eqref{lemAmax1} and Steps~2 and 3, if $m=0$ then $U\in B_0^0u$ and $-\e^{\vi\theta}\Delta u+a\e^{\vi\theta}U+b\e^{\vi\theta}|u|^{p-1}u+\gamma\e^{\vi\theta}u+u=F,$ in $\Dr^\p(\Omega),$ so in $L^2(\Omega),$ since $(u,U)\in D(A_0^0)\times B_0^0u.$ In other words, $u\in D(A_0^0)$ and $(I+A_0^0)u\ni F.$ In the same way, if $m>0$ then by Step~2, $u\in D(A_0^m)$ and $(I+A_0^m)u=F,$ in $L^2(\Omega).$
\medskip
\end{proof*}

\section{Proofs of the theorems of existence and uniqueness}
\label{proofexi}

\begin{vproof}{of Proposition~\ref{propsolL2}.}
Let $(u,f)$ be a weak solution and let $(u_n,f_n)_{n\in\N}$ be a sequence of strong solutions to \eqref{gl} satisfying \eqref{fn}--\eqref{cv}. By Remark~\ref{rmkp}, each $u_n$ satisfies~\eqref{rmkp1} with $H=L^2(\Omega).$ Set $Y=H^2_0(\Omega)\cap L^\frac2{2-m}(\Omega).$ Then $Y^\star=H^{-2}(\Omega)+L^\frac2m(\Omega),$ and by \cite[Lemma~6.2]{MR4053613}, we have for any $T>0,$
\begin{gather}
\label{dempropsolL2}
\Delta u_n\xrightarrow[n\to\infty]{C([0,T];H^{-2}(\Omega))}\Delta u \; \text{ and } \; g_0^m(u_n)\xrightarrow[n\to\infty]{C([0,T];L^\frac2m(\Omega))}g_0^m(u).
\end{gather}
It follows from the equation satisfied by each $u_n,$ \eqref{cv} and \eqref{dempropsolL2} that for any $T>0,$ $(u_n)_{n\in\N}$ is a Cauchy sequence in $W^{1,1}\big((0,T);Y^\star\big),$ so that \eqref{propsolL2} hold true. We use \eqref{cv}, \eqref{propsolL22} and \eqref{dempropsolL2} to pass to the limit in the equation satisfied by each $u_n.$ Then, $u$ satisfies \eqref{gl} in $L^1_\loc\big([0,\infty);Y^\star\big).$
\medskip
\end{vproof}

\begin{vproof}{of Proposition~\ref{propdep}.}
The embedding in \eqref{propdep1} comes from \cite[Lemma~A.4]{MR4053613} $(1<q<\infty)$ and \cite[Theorem~5.3]{MR4725781} $(q=1$ or $q=\infty).$ We make the difference between the two equations satisfied by $u$ and $\wt u,$ respectively. It follows that $u-\wt u$ satisfies the equation obtained in $L^1_\loc\big((0,\infty);X^\star\big).$ We take the $X^\star-X$ duality product with $\e^{-\vi\theta}(u-\wt u).$ By \cite[Lemma~A.5]{MR4053613}, \cite[Theorem~5.3, Lemma~6.1]{MR4725781}, Lemma~\ref{lemg} and Cauchy-Schwarz' inequality, we get,
\begin{gather*}
\frac12\frac\d{\d t}\|u-\wt u\|_{L^2(\Omega)}^2\le\|f-\wt f\|_{L^2(\Omega)}\|u-\wt u\|_{L^2(\Omega)},
\end{gather*}
almost everywhere on $(0,\infty).$ Integrating over $(s,t),$ we obtain \eqref{propdep2}. Finally, we note that the strong solutions satisfy \eqref{propdep1} with $q=m+1$ or $q=p+1,$ and that \eqref{propdep2} is stable by passing to the limit in $C\big([0,T];L^2(\Omega)\big)\times L^1\big((0,T);L^2(\Omega)),$ for any $T>0.$ By using \eqref{cv}, we then deduce that \eqref{propdep2} still holds true for the weak solutions.
\medskip
\end{vproof}

\begin{vproof}{of Theorem~\ref{thmstrongH2}.}
Let the assumptions of the theorem be fulfilled. With the help of Proposition~\ref{propAmax}, we may apply Barbu~\cite[Theorem~4.5, p.141]{MR2582280} (and also Vrabie~\cite[Theorem~1.7.1, p.23]{MR1375237}). It follows that there exists a unique $u\in W^{1,\infty}_\loc\big([0,\infty);L^2(\Omega)\big)$ satisfying $u(t)\in D(A_0^m)$ and \eqref{gl} in $L^2(\Omega),$ for almost every $t>0,$ $u(0)=u_0$ and \eqref{strongH23}. This last estimate yields \eqref{strongH21}. Let us take the $L^2$-scalar product of \eqref{gl} with $\e^{-\vi\theta}u.$ Since $u\in W^{1,\infty}_\loc\big([0,\infty);L^2(\Omega)\big),$ we may apply \cite[Lemma~A.5]{MR4053613} to obtain \eqref{L2}, for a.e.\,$t>0.$ By \eqref{L2}, the Cauchy-Schwarz inequality and integration, we get that
\begin{gather}
\label{demH21}
\forall t\ge0, \; \|u(t)\|_{L^2(\Omega)}\le\|u_0\|_{L^2(\Omega)}+\vint_0^t\|f(s)\|_{L^2(\Omega)}\d s,		\\
\begin{split}
\label{demH22}
\cos\theta\|\nabla u(t)\|_{L^2(\Omega)}^2+\Re(a\e^{\vi\theta})\|u(t)\|_{L^{m+1}(\Omega)}^{m+1}+&\Re(b\e^{\vi\theta})\|u(t)\|_{L^{p+1}(\Omega)}^{p+1}	\\
\le&\left(\|u_t(t)\|_{L^2(\Omega)}+\|f(t)\|_{L^2(\Omega)}\right)\|u(t)\|_{L^2(\Omega)},
\end{split}
\end{gather}
for a.e.\,$t>0.$ We have by \eqref{demH21}--\eqref{demH22} that, $u\in L^\infty_\loc\big([0,\infty);H^1_0(\Omega)\cap L^{m+1}(\Omega)\cap L^{p+1}(\Omega)\big)$ and that $u$ is an $H^2$-solution. In particular, $u\in C_\w\big([0,\infty);H^1_0(\Omega)\big).$ Now, we take the $L^2$-scalar product of \eqref{gl} with $-\e^{-\vi\theta}\Delta u.$ Using Lemma~\ref{lemenereg} and Cauchy-Schwarz' inequality, we get that
\begin{gather}
\label{demH23}
\cos\theta\|\Delta u\|_{L^2(\Omega)}\le\left(\|u_t\|_{L^2(\Omega)}+\|f\|_{L^2(\Omega)}\right),
\end{gather}
a.e.\,on $(0,\infty).$ Therefore, there exists $C>0$ such that for a.e.\,$t>0,$ $\|\Delta u(t)\|_{L^2(\Omega)}\le C,$ and since $\Delta u\in C_\w\big([0,\infty);H^{-1}(\Omega)\big),$ we infer that $\Delta u\in C_\w\big([0,\infty);L^2(\Omega)\big).$ Taking the $L^2$-scalar product of~\eqref{gl} with $\e^{-\vi\theta}g_0^p(u),$ we obtain,
\begin{gather*}
\begin{split}
&\Re\left(\,\vint_\Omega u_t\ovl{g_0^p(u)}\d x\right)+\Re\left(\e^{\vi\theta}\vint_\Omega(-\Delta u)\ovl{g_0^p(u)}\d x\right)+\Re\left(a\e^{\vi\theta}
		\vint_\Omega g_0^m(u)\ovl{g_0^p(u)}\d x\right)																	\\
+\;&\Re(b\e^{\vi\theta})\|u\|_{L^{2p}(\Omega)}^{2p}+\left(\Re(\gamma\e^{\vi\theta})+1\right)\vint_\Omega|u|^{p+1}d x
		=\Re\left(\e^{\vi\theta}\vint_\Omega f(t)\ovl{g_0^p(u)}\d x\right).
\end{split}
\end{gather*}
But $\Re(a\e^{\vi\theta})\ge0$ and $\int_\Omega g_0^m(u)\ovl{g_0^p(u)}\d x=\int_\Omega|u|^{m+p}\d x.$ We then deduce from the above estimates and the Cauchy-Schwarz and Young inequalities that
\begin{gather}
\label{demH231}
\|u(t)\|_{L^{2p}(\Omega)}^{2p}\le C\left(\|u_t(t)\|_{L^2(\Omega)}+\|f(t)\|_{L^2(\Omega)}\right).
\end{gather}
for a.e.\,$t>0.$ In particular, $u\in L^\infty_\loc\big([0,\infty);L^{2p}(\Omega)\big)$ and we infer with the help of \eqref{gl} that if $m>0$ then $u\in L^\infty_\loc\big([0,\infty);L^{2m}(\Omega)\big),$ and if $m=0$ then the saturated section $U$ asssociated to $u$ satisfies \eqref{thmstrongaH13}. Since we have for any $t\ge0,$
\begin{gather}
\label{demH24}
\|\nabla u(t)\|_{L^2(\Omega)}^2\le\|u(t)\|_{L^2(\Omega)}\|\Delta u(t)\|_{L^2(\Omega)},
\end{gather}
we get \eqref{strongH22} with the help of \eqref{strongH21} and \eqref{demH23}. In particular, $u\in C\big([0,\infty);H^1_0(\Omega)\big).$ Taking the $L^2$-scalar product of \eqref{gl} with $-\e^{-\vi\theta}\Delta u,$ it follows from~\cite[Lemma~A.5]{MR4053613}, Lemma~\ref{lemenereg} and the Cauchy-Schwarz and Young inequalities that for almost every $t>0,$
\begin{align*}
	&	\; \frac12\frac{\d}{\d t}\|\nabla u(t)\|_{L^2(\Omega)}^2+\cos\theta\|\Delta u(t)\|_{L^2(\Omega)}^2			\\
   \le	&	\; \|f(t)\|_{L^2(\Omega)}\|\Delta u(t)\|_{L^2(\Omega)}\le\frac1{2\cos\theta}\|f(t)\|_{L^2(\Omega)}^2
											+\frac{\cos\theta}2\|\Delta u(t)\|_{L^2(\Omega)}^2,
\end{align*}
from which we get Property~\ref{thmstrongH24}. It remains to prove Property~\ref{thmstrongH25}. This comes from \eqref{gl}, \eqref{strongH23}, \eqref{demH21} and \eqref{demH23}--\eqref{demH24}. This ends the proof of the theorem.
\medskip
\end{vproof}

\begin{lem}
\label{lemint}
Let Assumption~$\ref{ass}$ be fulfilled, let $(u,f)$ and $(u_n,f_n)_{n\in\N}$ satisfy \eqref{cv}, where $(u_n,f_n)_{n\in\N}$ $((u_n,U_n,f_n)_{n\in\N},$ if $m=0)$ is any sequence of $H^2$-strong solutions to~\eqref{gl}--\eqref{glb}. Then $u$ is a weak solution to \eqref{gl}--\eqref{glb} and satisfies \eqref{H10}, \eqref{Lm} as well as
\begin{align}
\label{lemint1}
&	u_n\underset{n\to\infty}{-\!\!\!-\!\!\!-\!\!\!-\!\!\!\weak}u, \text{ in } L^2\big((0,T);H^1_0(\Omega)\big)_\w,				\\
\label{lemint2}
&	\Delta u_n\underset{n\to\infty}{-\!\!\!-\!\!\!-\!\!\!-\!\!\!\weak}\Delta u, \text{ in } L^2\big((0,T);H^{-1}(\Omega)\big)_\w,	\\
\label{lemint3}
&	g_0^q(u_n)\underset{n\to\infty}{-\!\!\!-\!\!\!-\!\!\!-\!\!\!\weak}g_0^q(u), \text{ in } L^\frac{q+1}q\big((0,T);L^\frac{q+1}q(\Omega)\big)_\w, \; q>0,	\\
\label{lemint4}
&	U_n\underset{n\tends\infty}{\overset{L^\infty((0,T)\times\Omega)_{\w\star}}{-\!\!\!-\!\!\!-\!\!\!-\!\!\!-\!\!\!-\!\!\!-\!\!\!-\!\!\!-\!\!\!-\!\!\!-\!\!\!\weak}}U,
	\text{ if } m=0,	\\
\label{lemint5}
&	\; \frac{\partial u_n}{\partial t}\xrightarrow[n\tends\infty]{}\frac{\partial u}{\partial t}, \text{ in } \Dr^\p\big((0,T)\times\Omega\big),
\end{align}
for any $T>0$ and $q\in\{m,p\},$ where $U$ is the saturated section associated to $u$ when $m=0.$ Finally, $u$ satisfies \eqref{L2+}, for $s=0$ and any $t\ge0.$
\end{lem}

\begin{proof*}
Let $(u,f)$ and $(u_n,f_n)_{n\in\N}$ be as in the statement of the lemma. Let $T>0$ and $q\in\{m,p\}.$ By \eqref{cv} and the diagonal procedure, there exists a subsequence $(u_{n_k})_{k\in\N}\subset(u_n)_{n\in\N}$ satisfying
\begin{gather}
\label{demlemint1}
u_{n_k}\xrightarrow[k\to\infty]{\text{a.e.\,in }(0,\infty)\times\Omega}u,
\end{gather}
from which we get
\begin{gather}
\label{demlemint2}
g_0^q(u_{n_k})\xrightarrow[k\to\infty]{\text{a.e.\,in }(0,\infty)\times\Omega}g_0^q(u), \; q>0,			\\
\label{demlemint3}
U_{n_k}=\frac{u_{n_k}}{|u_{n_k}|}\xrightarrow[k\to\infty]{\text{a.e.\,in }\omega}\frac{u}{|u|}, \text{ if } m=0,
\end{gather}
where $\omega=\big\{(t,x)\in(0,\infty)\times\Omega;u(t,x)\neq0\big\}.$ By \eqref{L2}, we have for any $n\in\N,$
\begin{multline}
\label{L2n+}
\dfrac12\|u_n(T)\|_{L^2}^2+\cos\theta\vint_0^T\|\nabla u_n(t)\|_{L^2}^2\d t
+\Re(a\e^{\vi\theta})\vint_0^T\|u_n(t)\|_{L^{m+1}}^{m+1}\d t+\Re(b\e^{\vi\theta})\vint_0^T\|u_n(t)\|_{L^{p+1}}^{p+1}\d t	\\
+\Re(\gamma\e^{\vi\theta})\vint_0^T\|u_n(t)\|_{L^2}^2\d t
=\dfrac12\|u_n(0)\|_{L^2}^2+\Re\left(\e^{\vi\theta}\iint\limits_{0\;\Omega}^{\text{}\;\;T}f_n(t,x)\,\ovl{u_n(t,x)}\,\d x\,\d t\right).
\end{multline}
We infer with the help of \eqref{cv} and Cauchy-Schwarz' inequality that
\begin{gather}
\label{demlemin4}
(u_n)_{n\in\N} \text{ is bounded in } L^2\big((0,T);H^1_0(\Omega)\big),		\\
\label{demlemin5}
(u_n)_{n\in\N} \text{ is bounded in } L^{q+1}\big((0,T);L^{q+1}(\Omega)\big),			\\
\label{demlemin6}
(g_0^q(u_n))_{n\in\N} \text{ is bounded in } L^\frac{q+1}q\big((0,T);L^\frac{q+1}q(\Omega)\big), \; q>0.
\end{gather}
We claim that $u$ satisfies \eqref{H10}, \eqref{Lm}, \eqref{lemint1}, \eqref{lemint2}, \eqref{lemint3} and
\begin{align}
\label{demlemin7}
&	\nabla u_n\underset{n\to\infty}{-\!\!\!-\!\!\!-\!\!\!-\!\!\!\weak}\nabla u, \text{ in } L^2\big((0,T);L^2(\Omega)\big)_\w^N,	\\
\label{demlemin8}
&	\int_0^T\|u(t)\|_{L^{q+1}(\Omega)}^{q+1}\d t\le\liminf_{n\to\infty}\int_0^T\|u_n(t)\|_{L^{q+1}(\Omega)}^{q+1}\d t,
\end{align}
for the whole sequence. Indeed, first of all, since for any $r\in[1,\infty),$
\begin{gather}
\label{demlemin9}
L^r\big((0,T);L^r(\Omega)\big)\cong L^r((0,T)\times\Omega),
\end{gather}
we have by \eqref{demlemint1}, \eqref{demlemin5}, \eqref{demlemin9} and Fatou's Lemma that $u$ satisfies \eqref{Lm} and \eqref{demlemin8}. Next, we notice that 
\begin{gather}
\label{demlemin10}
C\big([0,T];L^2(\Omega)\big)\inj \Dr^\p\big((0,T)\times\Omega\big),
\end{gather}
and that the spaces involved in \eqref{lemint1} and \eqref{lemint3} are reflexive and also continuously embedded in $\Dr^\p\big((0,T)\times\Omega\big).$ We deduce from \eqref{cv}, \eqref{demlemin4} and \eqref{demlemin6} that \eqref{H10} holds true and, for a subsequence, that $u$ satisfies \eqref{lemint1} and \eqref{lemint3}. But if \eqref{lemint1} does not hold for the whole sequence, by density of $L^2\big((0,T);L^2(\Omega)\big)$ in $L^2\big((0,T);H^{-1}(\Omega)\big),$ we obtain, for a subsequence, that $u_n\cancel{-\!\!\!\weak}u,$ as $n\to\infty,$ in $L^2\big((0,T);L^2(\Omega)\big)_\w,$ which contradicts \eqref{cv}. Hence \eqref{lemint1} holds for the whole sequence, from which \eqref{lemint2} and \eqref{demlemin7} follow. In addition, if \eqref{lemint3} does not hold for the whole sequence then by \eqref{demlemin9}, there exists a sequence $(n_\ell)_{\ell\in\N}\subset\N$ for which we have that $g_0^q(u_{n_\ell})\cancel{-\!\!\!\weak}g_0^q(u),$ as $\ell\to\infty,$ in $L^{q+1}\big((0,T)\times\Omega\big)_\w.$ Still by \eqref{cv} and the diagonal procedure, there exists a subsequence $(n_{\ell_j})_{j\in\N}\subset(n_\ell)_{\ell\in\N}$ such that $u_{n_{\ell_j}}\tends u,$ as $j\to\infty,$ a.e.\,in $(0,\infty)\times\Omega,$ and so we have $g_0^q(u_{n_{\ell_j}})\tends g_0^q(u),$ as $j\to\infty$, a.e.\,in $(0,\infty)\times\Omega.$ This yields, with \eqref{demlemin6} that $g_0^q(u_{n_{\ell_j}})-\!\!\!-\!\!\!\!\weak g_0^q(u),$ as $j\to\infty,$ in $L^{q+1}\big((0,T)\times\Omega\big)_\w,$ a contradiction. Therefore \eqref{lemint3} holds true. Hence the claim. Moreoever, \eqref{lemint5} comes from \eqref{cv} and \eqref{demlemin10}. Finally, by \eqref{cv}, \eqref{demlemin7}, \eqref{demlemin8} and the weak lower semicontinuity of the norm, we obtain \eqref{L2+} for any $t=T>0$ and $s=0,$ with the help of \eqref{L2n+}. Now, assume that $m=0$ and let us establish \eqref{lemint4} for a subsequence. Since $\sup_{n\in\N}\|U_n\|_{L^\infty((0,\infty)\times\Omega)}\le1,$ there exists $U\in L^\infty((0,\infty)\times\Omega)$ such that, renumembering $(n_k)_{k\in\N}$ if necessary, we have
\begin{gather}
\label{demlemin11}
U_{n_k}\underset{k\tends\infty}{\overset{L^\infty((0,T)\times\Omega)_{\w\star}}{-\!\!\!-\!\!\!-\!\!\!-\!\!\!-\!\!\!-\!\!\!-\!\!\!-\!\!\!-\!\!\!-\!\!\!-\!\!\!\weak}}U.
\end{gather}
By \eqref{demlemint3} and \eqref{demlemin11}, we get that $U$ is a saturated section associated to $u$ (\cite[Lemma~6.1]{MR4848642}). Now, we return to the general case $m\in[0,1].$ Using \eqref{cv} and \eqref{lemint2}--\eqref{lemint5} with the sequence $(n_k)_{k\in\N}$ in \eqref{demlemin11} to pass to the limit as $k\to\infty$ in the equation \eqref{gl} satisfied by $(u_{n_k})_{k\in\N},$ we obtain that $(u,f)$ (or $(u,U,f),$ if $m=0)$ satisfied \eqref{gl} in $\Dr^\p\big((0,\infty)\times\Omega\big),$ so that $u$ is a weak solution to \eqref{gl}--\eqref{glb}. Finally, assume that $m=0.$ Writing that for any $n\in\N,$
\begin{gather*}
U_n-U=\e^{-\vi\theta}\left(\frac{\partial u}{\partial t}-\frac{\partial u_n}{\partial t}\right)-(\Delta u-\Delta u_n)
+b(g_0^p(u)-g_0^p(u_n))+\gamma(u-u_n)-(f_n-f)
\end{gather*}
and passing to the limit in $\Dr^\p\big((0,\infty)\times\Omega\big),$ we again obtain thanks to \eqref{cv}, \eqref{lemint1}--\eqref{lemint3} and \eqref{lemint5}, that $U_n\xrightarrow[n\to\infty]{}U,$ in $\Dr^\p\big((0,\infty)\times\Omega\big),$ and therefore \eqref{lemint4} is also true. The proof is ended.
\medskip
\end{proof*}

\begin{vproof}{of Theorem~\ref{thmweak}.}
Apply the densities of $\Dr\big((0,\infty);L^2(\Omega)\big)$ in $L^1_\loc([0,\infty);L^2(\Omega))$ and $\Dr(\Omega)$ in $L^2(\Omega),$ Theorem~\ref{thmstrongH2}, Proposition~\ref{propdep}, completeness of $C\big([0,T];L^2(\Omega)\big)$ and Lemma~\ref{lemint}.
\medskip
\end{vproof}

\begin{vproof}{of Theorem~\ref{thmweaksol}.}
Let $(u,f)$ be a weak solution and let $(u_n,f_n)_{n\in\N}$ be a any sequence of $H^2$-strong solutions to \eqref{gl} satisfying \eqref{fn}--\eqref{cv}. We first establish \eqref{L2+}. By Lemma~\ref{lemint}, \eqref{L2+} holds on any time interval $[0,T],$ $T\ge0.$ Let $t\ge s\ge0.$ For $\sigma>0,$ let $g(\sigma)=f(\sigma+s)$ and let $(v,g)$ be the weak solution to \eqref{gl} such that $v(0)=u(s).$ By uniqueness of weak solutions, $v(\sigma)=u(\sigma+s),$ for any $\sigma\ge0.$ Applying \eqref{L2+} on the time interval $[0,t-s]$ to $(v,g),$ we obtain the general case \eqref{L2+}. By Lemma~\ref{lemint}, it remains to prove \eqref{W11}--\eqref{senseq}. Let $T>0$ and let $v\in L^\infty\big((0,T);Y_{m,p}\big).$ It follows from \eqref{gl}, \eqref{cv}, \eqref{lemint2} and \eqref{lemint3} that
\begin{align*}
	&	\; \lim_{n\to\infty}\left\langle\e^{-\vi\theta}\frac{\partial u_n}{\partial t},v\right\rangle_{L^1((0,T);Y_{m,p}^\star),L^\infty((0,T);Y_{m,p})}
			=\langle\Delta u,v\rangle_{L^1((0,T);H^{-1}),L^\infty((0,T);H^1_0)}		\\
   -	&	\; \langle ag_0^m(u),v\rangle_{L^1((0,T);L^\frac{m+1}m),L^\infty((0,T);L^{m+1})}
   			-\langle bg_0^p(u),v\rangle_{L^1((0,T);L^\frac{p+1}p),L^\infty((0,T);L^{p+1})}	\\
   -	&	\; \langle\gamma u,v\rangle_{L^1((0,T);L^2),L^\infty((0,T);L^2)}
   			+\langle f,v\rangle_{L^1((0,T);L^2),L^\infty((0,T);L^2)},
\end{align*}
from which \eqref{W11}--\eqref{senseq} follow, by embedding $L^1\big((0,T);Y_{m,p}^\star\big)\inj \Dr^\p\big((0,T)\times\Omega\big)$ and \eqref{lemint5}.
\medskip
\end{vproof}

\begin{vproof}{of Theorem~\ref{thmweaksols}.}
Let $(u,U,f)$ be a weak solution and let $(u_n,U_n,f_n)_{n\in\N}$ be any sequence of strong solutions to \eqref{gl} satisfying \eqref{fn}--\eqref{cv}. Let $(Y_\ell)_{\ell\in\N_0}$ be any $L^1$-approximating sequence of RNP-spaces (\cite[Lemma~5.8]{MR4725781}). Let $T>0$ and let $\ell\in\N.$ Since $Y_\ell\inj L^1(\Omega)$ with dense embedding, it follows that
\begin{gather}
\label{demthmweaksols2}
L^\infty\big((0,T);Y_\ell)\big)\inj L^1\big((0,T);Y_\ell)\big)\inj L^1\big((0,T);L^1(\Omega)\big)\cong L^1\big((0,T)\times\Omega\big),
\end{gather}
with dense embeddings. By duality, since $Y_\ell$ is reflexive, we obtain that 
\begin{gather}
\label{demthmweaksols3}
L^\infty\big((0,T)\times\Omega\big)\inj L^1\big((0,T);Y_\ell)\big)^\star\cong L^\infty\big((0,T);Y_\ell^\star\big)\inj L^1\big((0,T);Y_\ell^\star\big),	\\
\label{demthmweaksols4}
L^1\big((0,T);Y_\ell^\star)\big)^\star\cong L^\infty\big((0,T);Y_\ell\big).
\end{gather}
It follows that for any $v\in L^\infty\big((0,T);Y_\ell)$ and $n\in\N,$
\begin{gather}
\left\langle U_n-U,v\right\rangle_{L^\infty((0,T)\times\Omega),L^1((0,T)\times\Omega)}
=\left\langle U_n-U,v\right\rangle_{L^1((0,T);Y_\ell^\star),L^\infty((0,T);Y_\ell)}
\end{gather}
from which we deduce that with the help of \eqref{lemint4} that
\begin{gather}
\label{demthmweaksols5}
U_n\underset{n\tends\infty}{\overset{L^1((0,T);Y_\ell^\star))_\w}{-\!\!\!-\!\!\!-\!\!\!-\!\!\!-\!\!\!-\!\!\!-\!\!\!-\!\!\!-\!\!\!-\!\!\!-\!\!\!\weak}}U.
\end{gather}
Using \eqref{lemint2}, \eqref{lemint3}, \eqref{lemint5}, \eqref{demthmweaksols5} and that for any $T>0,$ $L^1\big((0,T);Z_{\ell,p}^\star\big)\inj \Dr^\p\big((0,T)\times\Omega\big),$ we prove \eqref{W11l}--\eqref{senseqs} in the same way as for Theorem~\ref{thmweaksols}. Finally, \eqref{thmweaksols1} comes from \eqref{lemint4}, and \eqref{thmweaksols2} comes from the equation \eqref{gl} and Lemma~\ref{lemint}. Notice that $\Dr(\Omega)\inj Y_0$ with dense embedding and that $Y_\ell^\star\inj Y_0^\star,$ so that \eqref{W11l}--\eqref{senseqs} hold true for $\ell=0.$ Finally, if $|\Omega|<\infty$ then $H^1_0(\Omega)\cap L^1(\Omega)\cap L^{p+1}(\Omega)=H^1_0(\Omega)\cap L^{p+1}(\Omega),$ which is reflexive and separable. In this case, the above arguments work for $H^1_0(\Omega)\cap L^{p+1}(\Omega)$ in place of $Z_{\ell,p}.$
\medskip
\end{vproof}

\begin{vproof}{of Theorem~\ref{thmstrongaH1}.}
Let $u_0\in H^1_0(\Omega)$ and $f\in L^2_\loc\big([0,\infty);L^2(\Omega)\big).$ Let $u$ be the unique weak solution to \eqref{gl}--\eqref{u0} given by Theorem~\ref{thmweak}. Let $(\vphi_n)_{n\in\N}\subset\Dr(\Omega)$ and $(f_n)_{n\in\N}\subset\Dr\big((0,\infty);L^2(\Omega)\big)$ be such that $\vphi_n\xrightarrow[n\to\infty]{H^1_0(\Omega)}u_0$ and $f_n\xrightarrow[n\to\infty]{L^2((0,T);L^2)}f,$ for any $T>0.$ For each $n\in\N,$ let $(u_n,f_n)$ be the unique $H^2$-solution to \eqref{gl} such that $u_n(0)=\vphi_n$ given by Theorem~\ref{thmstrongH2}. By Proposition~\ref{propdep}, we have for any $T>0,$
\begin{gather}
\label{demthmstrongH11}
u_n\xrightarrow[n\tends\infty]{C([0,T];L^2(\Omega))}u, \; \nabla u_n\xrightarrow[n\tends\infty]{C([0,T];H^{-1}(\Omega))^N}\nabla u
\; \text{ and } \; \Delta u_n\xrightarrow[n\tends\infty]{C([0,T];H^{-2}(\Omega))}\Delta u.
\end{gather}
Since each $u_n$ satisfies Property~\ref{thmstrongH24} of Theorem~\ref{thmstrongH2}, we get after integration that for any $n\in\N$ and $t\ge0,$
\begin{gather}
\label{demthmstrongH12}
\|\nabla u_n(t)\|_{L^2(\Omega)}^2+\cos\theta\vint_0^t\|\Delta u_n(s)\|_{L^2(\Omega)}^2\d s
\le\|\nabla\vphi_n\|_{L^2(\Omega)}^2+\frac1{\cos\theta}\vint_0^t\|f_n(s)\|_{L^2(\Omega)}^2\d s.
\end{gather}
We infer with the help of \eqref{demthmstrongH11} and \eqref{demthmstrongH12} that for any $T>0,$ $(u_n)_{n\in\N}$ and $(\Delta u_n)_{n\in\N}$ are bounded in $C\big([0,T];H^1_0(\Omega)\big)$ and in $L^2\big((0,T);L^2(\Omega)),$ respectively. It follows from \eqref{demthmstrongH11} that for any $T>0,$
\begin{gather}
\label{demthmstrongH13}
u\in C_\w\big([0,\infty);H^1_0(\Omega)\big) \text{ and } \Delta u\in L^2_\loc\big([0,\infty);L^2(\Omega)\big),			\\
\label{demthmstrongH14}
u_n\underset{n\to\infty}{-\!\!\!-\!\!\!-\!\!\!-\!\!\!\weak}u, \text{ in } L^\infty\big((0,T);H^1_0(\Omega)\big)_{\w\star},		\\
\label{demthmstrongH15}
\nabla u_n(t)\underset{n\to\infty}{-\!\!\!-\!\!\!-\!\!\!-\!\!\!\weak}\nabla u(t), \text{ in } L^2(\Omega)^N_\w, \; \forall t\in[0,T],	\\
\label{demthmstrongH16}
\Delta u_n\underset{n\to\infty}{-\!\!\!-\!\!\!-\!\!\!-\!\!\!\weak}\Delta u, \text{ in } L^2\big((0,T);L^2(\Omega)\big)_\w.
\end{gather}
Using \eqref{demthmstrongH15}, \eqref{demthmstrongH16}, and the weak lower semicontinuity of the norm to pass to the limit in \eqref{demthmstrongH12}, we obtain \eqref{thmstrongaH12} with $s=0,$ for any $t>0.$ Now, we fix $s>0.$ Let $(v,f(\,.\,+s))$ be the weak solution to \eqref{gl} such that $v(0)=u(s).$ By uniqueness of weak solutions, $v(t)=u(t+s),$ for any $t\ge0.$ We then obtain the general case \eqref{thmstrongaH12}. Now, if $m>0$ then by \eqref{Lm}, \eqref{demthmstrongH13} and \eqref{gl}, we get that $u\in W^{1,q}_\loc\big([0,\infty);L^2(\Omega)+L^\frac{m+1}m(\Omega)+L^\frac{p+1}p(\Omega)\big),$ where $q$ is as the statement of the theorem. Finally, if $m=0$ and $|\Omega|<\infty$ then $U\in L^\infty\big((0,\infty)\times\Omega)\big)\inj L^2_\loc\big([0,\infty);L^2(\Omega)\big),$ so that $U:[0,\infty)\tends L^2(\Omega)$ is measurable and satisfies \eqref{thmstrongaH13}. Using again \eqref{Lm}, \eqref{demthmstrongH13} and \eqref{gl}, we get the desired regularity for $u$ which allows to apply \cite[Lemma~A.5]{MR4053613}, from which \eqref{L2} follows.
\medskip
\end{vproof}

\begin{vproof}{of Theorem~\ref{thmstrongH1}.}
Let the assumptions of the theorem be fulfilled. Let $u_0\in H^1_0(\Omega)$ and let $u$ be the unique weak solution to~\eqref{gl}--\eqref{u0} given by Theorem~\ref{thmweak}. We begin with the case $b=0.$ By \eqref{Lm}, \eqref{thmstrongaH11}, \eqref{thmstrongaH13} and \eqref{gl}, we have that $u\in W^{1,\frac{m+1}m}_\loc\big([0,\infty);X^\star\big),$ where $X=H^1_0(\Omega)\cap L^{m+1}(\Omega).$ Hence $u$ satisfies \eqref{defsol11} with $q=m$ and is an $H^1_0$-solution. If we have \eqref{thmstrongH112} with $m>0$ then by \eqref{thmstrongaH11} and Sobolev' embedding, we have for any $T>0,$
\begin{gather*}
\vint_0^T\|g_0^p(u(t))\|_{L^\frac{p+1}p(\Omega)}^\frac{m+1}m\d t=\vint_0^T\|u(t)\|_{L^{p+1}(\Omega)}^{p\frac{m+1}m}\d t
\le C(N,p)T\|u\|_{L^\infty((0,T);H^1_0(\Omega))}^{p\frac{m+1}m},
\end{gather*}
while if $m=0$ then $U\in L^\infty_\loc\big([0,\infty);L^2(\Omega)\big)$ by \eqref{thmstrongaH13}. Hence with \eqref{Lm}, \eqref{thmstrongaH11} and \eqref{gl}, $u$ satisfies \eqref{defsol11} with $q=m$ and is an $H^1_0$-solution. Now, assume that we have \eqref{thmstrongH113}. If $m>0$ then we have by H\"{o}lder's inequality that
\begin{gather*}
\vint_0^T\|g_0^m(u(t))\|_{L^\frac{p+1}p(\Omega)}^\frac{p+1}p\d t=\iint\limits_{(0,T)\times\Omega}|u(t,x)|^{m\frac{p+1}p}\d t\d x
\le (T|\Omega|)^\frac{p-m}p\left(\vint_0^T\|u(t)\|_{L^{p+1}(\Omega)}^{p+1}\d t\right)^\frac{m}p,
\end{gather*}
for any $T>0,$ while if $m=0$ then $U\in L^\infty_\loc\big([0,\infty);L^2(\Omega)\big)$ by \eqref{thmstrongaH13}. We infer with \eqref{Lm}, \eqref{thmstrongaH11}, \eqref{thmstrongaH13} and \eqref{gl}, that $u$ satisfies \eqref{defsol11} with $q=p$ and is an $H^1_0$-solution. So, in these three cases, $u$ is an $H^1_0$-solution, which is unique with the help of Proposition~\ref{propdep}. It remains to show that the map $t\longmapsto\|u(t)\|_{L^2(\Omega)}^2$ belongs to $W^{1,1}_\loc\big([0,\infty);\R\big)$ and that \eqref{L2} holds for almost every $t>0.$ Let $Y=H^1_0(\Omega)\cap L^{m+1}(\Omega)\cap L^{p+1}(\Omega).$ By the regularity of $u,$ we may take the $Y^\star-Y$ duality product of \eqref{gl} with $\e^{-\vi\theta}u.$ Then \cite[Lemma~A.5]{MR4053613} gives the desired result.
\medskip
\end{vproof}

\begin{rmk}
\label{rmkthmH1}
The assumptions $f\in L^2_\loc\big([0,\infty);L^2(\Omega)\big)$ in Theorem~\ref{thmstrongaH1}, and \eqref{thmstrongH11} in Theorem~\ref{thmstrongH1}, may be replaced with
\begin{gather*}
f\in L^1_\loc\big([0,\infty);H^1_0(\Omega)\big),		\\
f\in L^1_\loc\big([0,\infty);H^1_0(\Omega)\big)
\cap L^\frac{m+1}m_\loc\big([0,\infty);H^{-1}(\Omega)+L^\frac{m+1}m(\Omega)+L^\frac{p+1}p(\Omega)\big),
\end{gather*}
respectively, as for the $H^1_0$-solutions of the Schr\"{o}dinger equation in \cite{MR4340780,MR4503241,MR4725781} (which corresponds to $\left.\theta=\pm\frac\pi2 \text{ in }\eqref{gl}\right).$ In this case, \eqref{L2} is no more valid and \eqref{thmstrongaH12} has to replaced by
\begin{gather}
\label{rmkthmH11}
\begin{split}
&\frac12\|\nabla u(t)\|_{L^2(\Omega)}^2+\cos\theta\vint_s^t\|\Delta u(\sigma)\|_{L^2(\Omega)}^2\d\sigma			\\
\le\;&\frac12\|\nabla u(s)\|_{L^2(\Omega)}^2
	+\left(\|\nabla u(s)\|_{L^2(\Omega)}+\vint_s^t\|f(\sigma)\|_{L^2(\Omega)}\d\sigma\right)\vint_s^t\|f(\sigma)\|_{L^2(\Omega)}\d\sigma,
\end{split}
\end{gather}
for any $t\ge s\ge0.$ In addition, the regularity of $u$ is no more $W^{1,2}_\loc$ but $W^{1,1}_\loc,$ and the equation is no more satisfied in $L^2_\loc$ but in $L^1_\loc.$ We proceed as follows. We use the notation of the proof of Theorem~\ref{thmstrongaH1}, but with the assumption that $(f_n)_{n\in\N}\subset\Dr\big((0,\infty);H^1_0(\Omega)\big)$ with $f_n\xrightarrow[n\to\infty]{L^1((0,T);H^1_0)}f,$ for any $T>0.$ Taking the $L^2$-scalar product of \eqref{gl} with $-\e^{-\vi\theta}\Delta u_n,$ we obtain for any $n\in\N$ and a.e.\;$t>0$ (\cite[Lemma~A.5]{MR4053613} and Lemma~\ref{lemenereg}),
\begin{gather}
\label{rmkthmH12}
\frac12\frac{\d}{\d t}\|\nabla u_n(t)\|_{L^2(\Omega)}^2+\cos\theta\|\Delta u_n(t)\|_{L^2(\Omega)}^2
\le\|\nabla f_n(t)\|_{L^2(\Omega)}\|\nabla u_n(t)\|_{L^2(\Omega)},
\end{gather}
from which we get,
\begin{gather}
\label{rmkthmH13}
\|\nabla u_n(t)\|_{L^2(\Omega)}\le\|\nabla\vphi_n\|_{L^2(\Omega)}+\vint_0^t\|\nabla f_n(s)\|_{L^2(\Omega)}\d s,
\end{gather}
for any $n\in\N$ and $t\ge0.$ Putting \eqref{rmkthmH13} in \eqref{rmkthmH12} and integrating the result, we get
\begin{gather}
\label{rmkthmH14}
\begin{split}
&\frac12\|\nabla u_n(t)\|_{L^2(\Omega)}^2+\cos\theta\vint_0^t\|\Delta u_n(s)\|_{L^2(\Omega)}^2\d s			\\
\le\;&\frac12\|\nabla\vphi_n\|_{L^2(\Omega)}^2
	+\left(\|\nabla\vphi_n\|_{L^2(\Omega)}+\vint_0^t\|f_n(s)\|_{L^2(\Omega)}\d s\right)\vint_0^t\|f_n(s)\|_{L^2(\Omega)}\d s,
\end{split}
\end{gather}
for any $n\in\N$ and $t\ge0,$ instead of \eqref{demthmstrongH14}. Finally, we proceed as in the proof of Theorem~\ref{thmstrongaH1} to obtain \eqref{rmkthmH11}, whereas the proof of Theorem~\ref{thmstrongH1} is unchanged. In the same way, we get from \eqref{rmkthmH13} that
\begin{gather}
\label{rmkthmH15}
\|\nabla u(t)\|_{L^2(\Omega)}\le\|\nabla u(s)\|_{L^2(\Omega)}+\vint_s^t\|\nabla f(\sigma)\|_{L^2(\Omega)}\d\sigma,
\end{gather}
for any $t\ge s\ge0.$
\end{rmk}

\baselineskip .4cm


\def\cprime{$^\prime$}

\end{document}